\documentclass[12pt, a4paper]{article}
\usepackage{amsmath}
\usepackage{comment}
\usepackage{amsfonts}
\usepackage{amssymb}
\usepackage{epsfig}
\usepackage{graphicx}
\usepackage{color}
\usepackage{amsthm}
\usepackage{enumerate}
\usepackage [latin1]{inputenc}
 \usepackage[numbers,sort&compress]{natbib}

\newtheorem{thm}{Theorem}[section]
\newtheorem{lem}[thm]{Lemma}
\newtheorem{cor}[thm]{Corollary}
\newtheorem{pro}[thm]{Proposition}

\theoremstyle{definition}

\newtheorem{exa}[thm]{Example}

\graphicspath{{figures/}}

\usepackage{url}

\usepackage{authblk}

\long\def\delete#1{}

\usepackage{xcolor}
\usepackage[normalem]{ulem}

\input amssym.def
\input amssym.tex

\newcommand{\pmat}[1]{\begin{pmatrix}#1\end{pmatrix}}

\newcommand{\cay}{\mathrm{Cay}}

\newcommand{\zzz}{\mathbb{Z}_2^n}
\newcommand{\zp}{\mathbb{Z}_p^n}
\newcommand{\dn}{D_{2n}}
\newcommand{\bn}{Q_{4n}}

\def\FFF{\mathbb{F}}
\def\ZZZ{\mathbb{Z}}

\def\be{\mathbf{e}}
\def\bu{\mathbf{u}}
\def\bv{\mathbf{v}}
\def\bw{\mathbf{w}}

\def\Ga{\Gamma}

\def\Si{\Sigma}

\begin{document}
\openup 0.5\jot

\title{On regular sets in Cayley graphs}

\author[1]{Xiaomeng Wang\thanks{E-mail: \texttt{wangxm2015@lzu.edu.cn}}}
\author[1]{Shou-Jun Xu\thanks{E-mail: \texttt{shjxu@lzu.edu.cn}}}
\author[2]{Sanming Zhou\thanks{E-mail: \texttt{sanming@unimelb.edu.au}}}

\affil[1]{\small School of Mathematics and Statistics, Gansu Center for Applied Mathematics, Lanzhou University, Lanzhou, Gansu 730000, China}
\affil[2]{\small School of Mathematics and Statistics, The University of Melbourne, Parkville, VIC 3010, Australia}

\date{}

\maketitle

\begin{abstract}
Let $\Ga = (V, E)$ be a graph and $a, b$ nonnegative integers. An $(a, b)$-regular set in $\Ga$ is a nonempty proper subset $D$ of $V$ such that every vertex in $D$ has exactly $a$ neighbours in $D$ and every vertex in $V \setminus D$ has exactly $b$ neighbours in $D$. A $(0,1)$-regular set is called a perfect code, an efficient dominating set, or an independent perfect dominating set. A subset $D$ of a group $G$ is called an $(a,b)$-regular set of $G$ if it is an $(a, b)$-regular set in some Cayley graph of $G$, and an $(a, b)$-regular set in a Cayley graph of $G$ is called a subgroup $(a, b)$-regular set if it is also a subgroup of $G$. In this paper we study $(a, b)$-regular sets in Cayley graphs with a focus on $(0, k)$-regular sets, where $k \ge 1$ is an integer. Among other things we determine when a non-trivial proper normal subgroup of a group is a $(0, k)$-regular set of the group. We also determine all subgroup $(0, k)$-regular sets of dihedral groups and generalized quaternion groups. We obtain necessary and sufficient conditions for a hypercube or the Cartesian product of $n$ copies of the cycle of length $p$ to admit $(0, k)$-regular sets, where $p$ is an odd prime and a hypercube is a Cayley graph of an elementary abelian $2$-group. Our results generalize several known results from perfect codes to $(0, k)$-regular sets.


\end{abstract}

\section{Introduction}

All groups considered in this paper are finite, and all graphs considered are finite and simple. As usual, we use $V(\Gamma)$ and $E(\Ga)$ to denote the vertex set and edge set of a graph $\Gamma$, respectively. For a vertex $v \in V(\Gamma)$, the {\em neighbourhood} of $v$ in $\Gamma$, denoted by $N_{\Gamma}(v)$ or simply $N(v)$, is the set of vertices adjacent to $v$ in $\Gamma$. The subgraph of $\Gamma$ induced by a subset $S$ of $V(\Gamma)$, denoted by $\Gamma[S]$, is the graph with vertex set $S$ in which two vertices are adjacent if and only if they are adjacent in $\Ga$.

Let $a$ and $b$ be nonnegative integers. An {\em $(a,b)$-regular set}~\cite{cardoso19} in a graph $\Gamma$ is a nonempty proper subset $D$ of $V(\Gamma)$ such that $|N(v)\cap D|=a$ for every $v \in D$ and $|N(v)\cap D|=b$ for every $v \in V(\Gamma)\setminus D$. In particular, a $(0,1)$-regular set in $\Gamma$ is called a {\em perfect code}~\cite{Huang18, Ma19, Zhang20}, an {\em independent perfect dominating set}~\cite{hhs981, Lee01}, or an {\em efficient dominating set}~\cite{de03, hhs981}. In this paper, we study $(a, b)$-regular sets in Cayley graphs with a focus on the special case when $(a, b) = (0, k)$ for a positive integer $k$. Let $G$ be a group with identity element $e$. Set $X^{-1} = \{x^{-1}: x \in X\}$ for any $X \subseteq G$. If $X^{-1} = X$, then $X$ is said to be inverse-closed. For any inverse-closed subset $S$ of $G \setminus \{e\}$, the {\em Cayley graph} $\cay(G,S)$ of $G$ with {\em connection set} $S$ is the graph with vertex set $G$ and edge set $\{\{x, y\}\ :\ x,y \in G, yx^{-1} \in S\}$ \cite{Big01}. It is readily seen that $\cay(G,S)$ is a regular graph with degree $|S|$, and $\cay(G,S)$ is connected if and only if $S$ is a generating set of $G$.

In recent years, perfect codes in Cayley graphs have attracted considerable attention due to their connections with perfect $1$-error correcting codes in coding theory, tilings of groups, and equitable partitions of graphs (see, for example, \cite{Chen19,cr13,de03,deng14,deng17,feng17,Huang18,Kwo14,Lee01,Ma19,mar07,ob07,Rej13,Wang21}).
The reader is referred to \cite[Section 1]{Huang18} and \cite[Section 1]{Wang20} for such background information. In particular, the Hamming graph $H(n, q)$ and the Cartesian product $C_q^{\Box n}$ of $n$ copies of the cycle $C_q$ with length $q$ are both Cayley graphs of $\ZZZ_q^n$, and the Hamming and Lee metrics over $\ZZZ_q^n$ are exactly the graph distances in these graphs, respectively. Thus, perfect $1$-codes \cite{heden08} under these metrics are exactly perfect codes in $H(n, q)$ and $C_q^{\Box n}$, respectively. So perfect codes in Cayley graphs are an ample generalization of perfect $1$-codes in classical coding theory.

Huang et al.~\cite{Huang18} introduced the following concepts: A subset $D$ of a group $G$ is called a {\em perfect code} of $G$ if it is a perfect code in some Cayley graph of $G$, and a perfect code in a Cayley graph $\cay(G, S)$ is called a {\em subgroup perfect code} if it is also a subgroup of $G$. In \cite{Wang20, Wang21}, Wang, Xia and Zhou generalized these concepts to regular sets: A subset $D$ of a group $G$ is called an {\em $(a,b)$-regular set} of  $G$ if it is an $(a, b)$-regular set in some Cayley graph of  $G$, and an $(a,b)$-regular set in a Cayley graph $\cay(G, S)$ is called a {\em subgroup $(a,b)$-regular set} if it is also a subgroup of $G$. In~\cite{cr13}, Tamizh Chelvam and Mutharasu obtained a necessary and sufficient condition for a subgroup of a cyclic group to be a perfect code. In~\cite{Huang18}, Huang, Xia and Zhou gave a necessary and sufficient condition for a normal subgroup of a group $G$ to be a subgroup perfect code of $G$ and determined all subgroup perfect codes of dihedral groups and some abelian groups. Continuing this line of research, Chen, Wang and Xia \cite{Chen19} obtained necessary and sufficient conditions for a subgroup to be a subgroup perfect code of the group. In~\cite{Ma19}, Ma \emph{et al.} proved that a group $G$ admits every subgroup as a perfect code if and only if $G$ has no elements of order $4$. In \cite{Zhang20, Zhang22}, Zhang and Zhou obtained multiple results on subgroup perfect codes, including a few necessary and sufficient conditions for a subgroup to be a subgroup perfect code and several results on subgroup perfect codes of metabelian groups, generalized dihedral groups, nilpotent groups and 2-groups.

In~\cite{Bai20}, Bailey, Cameron and Zhou studied regular sets in Cayley graphs with the help of group representations in the general framework of equitable partitions. In~\cite{Wang20}, Wang, Xia and Zhou proved that for any normal subgroup $H$ of a group $G$, the following conditions are equivalent: (a) for any $g \in G$ with $g^2 \in H$, there exists $h \in H$ such that $(gh)^2 = e$; (b) $H$ is a perfect code of $G$; (c) $H$ is an $(a, b)$-regular set of $G$ for every pair of integers $a, b$ with $0\leq a\leq |H|-1$ and $0\leq b\leq |H|$ such that $\gcd(2, |H|-1)$ divides $a$. In the same paper they asked whether (a) and (b) are equivalent, and whether (b) and (c) are equivalent, if $H$ is not a normal subgroup of $G$. In \cite{Wang21}, they gave a positive answer to the latter question in the case when $G$ is a generalized dihedral group or a group of order $4p$ or $pq$ for some primes $p$ and $q$. With regard to the former question, an infinite family of counterexamples showing that (b) does not imply (a) when $H$ is non-normal was constructed by Behajaina, Maleki and Razafimahatratra in \cite{beh21}.

This paper is a study of regular sets in Cayley graphs with a focus on $(0,k)$-regular sets. The significance of $(0,k)$-regular sets in the study of regular sets can be seen from the following observations: Let $H$ be a subgroup $(0, k)$-regular set of a group $G$. Then there exists an inverse-closed subset $S$ of $G\setminus\{e\}$ such that $H$ is a $(0, k)$-regular set in $\cay(G,S)$. If $H$ contains no involutions, then for any even integer $a$ between $1$ and $|H|$, we can take a subset $S_0$ of $H\setminus\{e\}$ with size $a/2$ such that $H$ is an $(a,k)$-regular set in $\cay(G, S_0\cup S_0^{-1}\cup S)$. If $H$ contains involutions, then for any integer $a$ between $0$ and $|H|$, we can take subsets $S_0, S_1$ of $H \setminus \{e\}$ with $2|S_0|+|S_1|=a$ such that $S_0$ has no involutions, every element of $S_1$ is an involution, and $H$ is an $(a, k)$-regular set in $\cay(G, S_0\cup S_0^{-1}\cup S_1\cup S)$. On the other hand, if $H$ is a subgroup $(a, b)$-regular set of a group $G$, then $H$ is also a subgroup $(0, b)$-regular set of $G$ (but $H$ may not be a $(0, 1)$-regular set of $G$ as exemplified by $\{0, 2\}$, which is a subgroup $(0,2)$-regular set of $\ZZZ_4$ but not a $(0,1)$-regular set of $\ZZZ_4$).

The main results in this paper are as follows. In the next section we present some preliminary results that will be used in subsequent sections. Among other things we prove (Theorem~\ref{th:th2.0}) in Section \ref{sec:prel} that a Cayley graph of a group $G$ admits a proper subgroup $H$ of $G$ as a $(0,k)$-regular set if and only if it is a $k$-cover of $K_n$ with $H$ as a vertex-fiber, where $n = |G:H|$ is the index of $H$ in $G$. In Section \ref{sec:eq-partn}, we establish connections between subgroup $(0,k)$-regular sets in a Cayley graph and $(-k)$-equitable partitions of the graph (Theorem \ref{th:th4.11}). In Section \ref{sec:secnor}, we investigate when a non-trivial proper normal subgroup $H$ of a group $G$ is a subgroup $(0,k)$-regular set of $G$. We prove that this is always the case when $k$ is even (Theorem \ref{th:th5.3}), and we give a necessary and sufficient condition for $H$ to be a $(0,k)$-regular set of $G$ when $k$ is odd (Theorem \ref{th:th5.4}). As a consequence, we obtain that, for odd $k$, $H$ is a $(0,k)$-regular set of $G$ if and only if it is a $(0,1)$-regular set of $G$ (Corollary~\ref{co:co5.0}). In the case when $G$ is abelian and $k$ is odd, we give a necessary and sufficient condition for $H$ to be a $(0,k)$-regular set of $G$ in terms of the Sylow $2$-subgroups of $G$ (Theorem \ref{le:le5.1}). These results for odd $k$ generalize corresponding results in \cite{Huang18, Wang20} from $(0, 1)$-regular sets to $(0, k)$-regular sets. In Section \ref{sec:dih-qua}, we determine all subgroup $(0,k)$-regular sets of dihedral groups and generalized quaternion groups for all positive integers $k$ (Theorems \ref{th:th6.2} and~\ref{th:th6.4}), generalizing \cite[Theorem 2.11]{Huang18} and \cite[Theorem 1.7]{Ma19}, respectively, from $k=1$ to any $k$. In Section \ref{sec:seccube}, we prove that a hypercube $Q_n$ admits a $(0,k)$-regular set if and only if its degree is of the form $(2^t-1)k$ for some $t$ (Theorem~\ref{cor:Qn}), and under this condition we give a construction of linear $(0,k)$-regular sets in such a hypercube, where as usual the word ``linear" indicates that these sets are subspaces of $\FFF_2^n$. This is a generalization of the known result that perfect $1$-error correcting codes of length $n$ under the Hamming metric exist if and only if $n = 2^t-1$ for some $t$ (see, for example, \cite{heden08}). In the last section, Section \ref{sec:secp}, we prove that for any odd prime $p$ and integers $n, k \ge 1$, $C_p^{\Box n}$ admits a $(0,k)$-regular set if and only if its degree $2n$ is equal to $(p^t-1)k$ for some $t$ (Theorem~\ref{th:th5.2}), and under this condition we give a construction of linear $(0,k)$-regular sets in $C_p^{\Box n}$. This generalizes the known result that for any odd prime $p$, $p$-ary linear perfect $1$-error correcting codes of length $n$ under the Lee metric exist if and only if $2n = p^t - 1$ for some $t$, which can be obtained from \cite[Theorem 15]{ALB09} by choosing $q$ there to be an odd prime and $2n+1$ to be a prime power.

\section{Preliminaries}
\label{sec:prel}

As usual, for a group $G$ and a subgroup $H$ of $G$, we use $|G:H|$ to denote the index of $H$ in $G$ and $N_{G}(H)$ to denote the normalizer of $H$ in $G$. A \emph{left transversal} (respectively, \emph{right transversal}) of $H$ in $G$ is a subset of $G$ which contains exactly one element from each left coset (respectively, right coset) of $H$ in $G$. For any $D \subseteq G$ and $g \in G$, set
\[
gD = \{gx\ :\ x\in D\},\,\, Dg=\{xg\ :\ x\in D\}.
\]
The following lemma follows from the fact that the right regular representation of $G$ is a subgroup of the automorphism group of any Cayley graph of $G$.

\begin{lem}
\label{le:le2.0}
Let $G$ be a group, $D$ a proper subset of $G$, $S$ an inverse-closed subset of $G \setminus \{e\}$, and $k$ a positive integer. Then $D$ is a $(0,k)$-regular set in $\cay(G,S)$ if and only if $Dg$ is a $(0,k)$-regular set in $\cay(G,S)$ for every $g \in G$.
\end{lem}

Let $\Gamma$ be a $d$-regular graph and $D$ an $(a, b)$-regular set in $\Gamma$. By double counting the number of edges between $D$ and $V(\Gamma) \setminus D$, we obtain $|V(\Gamma)| = |D| \left(1+\frac{d-a}{b}\right)$ (see, for example, \cite{Wang20}). In particular, if $\Gamma$ admits a $(0, k)$-regular set $D$, where $k \le d$, then
\begin{equation}
|V(\Gamma)| = |D| \left(1+\frac{d}{k}\right)
\label{le:le1.1}
\end{equation}
and hence $k+d$ divides $k |V(\Ga)|$. Moreover, all $(0, k)$-regular sets in a regular graph have the same size. In particular, by (\ref{le:le1.1}), if $D$ is a $(0,k)$-regular set in a Cayley graph $\cay(G,S)$, then $k$ divides $|S| |D|$. Furthermore, in the case when $D$ is a subgroup of $G$, we have $|G:D| = 1 + \frac{|S|}{k}$ and hence $k$ divides $|S|$. So we have the following lemma.

\begin{lem}
\label{pr:pr2.1}
Let $G$ be a group, $S$ an inverse-closed subset of $G \setminus \{e\}$, and $k$ a positive integer. If $\cay(G,S)$ admits a subgroup $(0,k)$-regular set, then $k$ divides $|S|$.
\end{lem}

\begin{lem}
\label{le:le2.2}
Let $G$ be a group, $H$ a proper subgroup of $G$, $S$ an inverse-closed subset of $G \setminus \{e\}$, and $k$ a positive integer. If $H$ is a $(0,k)$-regular set in $\cay(G,S)$, then there exists a subset $L$ of $S$ with $|L| = |G:H| - 1$ such that $L \cup \{e\}$ is a left transversal of $H$ in $G$.
\end{lem}

\begin{proof}
Denote $\Ga = \cay(G,S)$ and $m = |G:H| - 1$. Since $H$ is a subgroup $(0,k)$-regular set in $\Gamma$, we have $|S| = km$ by \eqref{le:le1.1}. Since $k \ge 1$, we know that $G$ is the union of $sH$, for $s \in S \cup \{e\}$. Of course, for any two distinct elements $s, s'$ of $S$, we have either $sH= s'H$ or $sH\cap s'H=\emptyset$.

\smallskip
\textsf{Claim 1.} For any $k+1$ distinct elements $s_{1}, s_{2}, \ldots, s_{k+1}$ of $S$, we have $s_{1}H\cap s_{2}H\cap\cdots \cap s_{k+1}H=\emptyset$.
\smallskip

Suppose otherwise. Then there exists a vertex $u \in G$ such that $u=s_{1}h_1=s_{2}h_2=\cdots=s_{k+1}h_{k+1}$ for $k+1$ distinct elements $h_1, h_2,\ldots, h_{k+1}$ of $H$. That is, $u$ is adjacent in $\Gamma$ to $k+1$ distinct vertices in $H$, which contradicts the assumption that $H$ is a $(0,k)$-regular set in $\Gamma$.

\smallskip
\textsf{Claim 2.} For any $s \in S$, there exist exactly $k-1$ distinct elements $s_{1}, \ldots,  s_{k-1}$ of $S$ other than $s$ such that $s_1H = \cdots = s_{k-1}H = sH$.
\smallskip

In fact, let $\{s_{1}, \ldots,  s_{l}\}$ be a subset of $S$ with maximum size $l$ such that $s_1H = \cdots = s_{l}H = sH$. Then there exist distinct elements $h_1, \ldots, h_l, h$ of $H$ such that $s_1h_1 = \cdots = s_{l}h_l = sh$. By Claim 1, we have $0 \le l < k$. If $l < k - 1$, then $sh$ is adjacent to at most $l+1 < k$ vertices in $H$, which contradicts the assumption that $H$ is a $(0,k)$-regular set in $\cay(G,S)$. Hence $l = k-1$ and Claim 2 follows.

Recall that $G$ is the union of $sH$, for $s \in S \cup \{e\}$. This together with Claim 2 implies the existence of a subset $L$ of $S$ with $|L| = |G:H| - 1$ such that $L \cup \{e\}$ is a left transversal of $H$ in $G$.
\end{proof}

A graph $\Gamma$ is called a \emph{$k$-cover} of a graph $\Sigma$ (see, for example, \cite{ZZ22}) if there exists a surjective homomorphism $f$ from $\Gamma$ to $\Sigma$ such that $|N_{\Gamma}(u) \cap f^{-1}(y)| = k$ for any $\{x,y\}\in E(\Sigma)$ and $u \in f^{-1}(x)$. The homomorphism $f$ is called a \emph{$k$-covering} from $\Gamma$ to $\Sigma$, and if in addition $f$ is $m$-to-$1$ for some $m \ge k$ then $f$ is called an {\em $m$-fold $k$-covering}. Call $f^{-1}(x)$ the \emph{fiber} of $x \in V(\Si)$ and $f^{-1}(\{x,y\}) = \{\{u,v\} \in E(\Ga): \{f(u), f(v)\} = \{x, y\}\}$ the \emph{fiber} of $\{x,y \}\in E(\Si)$ under $f$. Of course, a $1$-cover is a cover and a $1$-covering is a covering in the usual sense (see \cite{Big01, god01, kwak98}). It is readily seen that for any $k$-covering $f:\Gamma \rightarrow \Si$, every vertex-fiber of $f$ is an independent set of $\Gamma$ and the subgraph of $\Gamma$ induced by any two distinct vertex-fibers is either an empty graph or a $k$-regular bipartite graph. (Here we use the assumption that $\Ga$ and $\Si$ are simple graphs.) In particular, if $\Si$ is connected, then all vertex-fibers have the same size, say, $m$, and hence $f$ is an $m$-fold $k$-covering.

The following lemma follows directly from the definitions of $(0,k)$-regular sets and $k$-coverings. Denote by $K_n$ the complete graph with $n$ vertices.

\begin{lem}
\label{le:le1.4}
Let $\Ga$ and $\Si$ be graphs and $k$ a positive integer. Then the following hold:
\begin{enumerate}[\rm (a)]
\item
if $f: \Gamma \rightarrow \Sigma$ is a $k$-covering and $D$ is a $(0,1)$-regular set in $\Sigma$, then $f^{-1}(D)$ is a $(0,k)$-regular set in $\Gamma$;
\item
if $\Gamma$ is $d$-regular, where $d \ge 1$, and $D_1, D_2, \ldots,D_n$ are pairwise disjoint $(0,k)$-regular sets in $\Gamma$, then $\Gamma[D_1\cup D_2 \cup \cdots \cup D_n]$ is an $m$-fold $k$-cover of $K_n$, where $m = \frac{k |V(\Gamma)|}{k+d}$ is the common size of $D_i$, for $1\leq i\leq n$.
\end{enumerate}
\end{lem}

The following result generalizes \cite[Theorem 1]{Lee01} from $(0,1)$-regular sets to $(0,k)$-regular sets.

\begin{thm}
\label{th:th1.1}
A regular graph $\Gamma$ is a $k$-cover of $K_n$ if and only if $V(\Gamma)$ can be partitioned into $n$ blocks each of which is a $(0,k)$-regular set in $\Gamma$, where $n$ and $k$ are positive integers.
\end{thm}

\begin{proof}
The sufficiency follows from part (b) of Lemma~\ref{le:le1.4}. For the necessity, suppose that $\Gamma$ is a $k$-cover of $K_n$, with $f: \Gamma \rightarrow K_n$ a $k$-covering. Then by part (a) of Lemma~\ref{le:le1.4}, the fiber $f^{-1}(u)$ of any $u \in V(K_n)$ is a $(0,k)$-regular set in $\Gamma$. Since $f$ is surjective, $\{f^{-1}(u): u\in V(K_n)\}$ is a partition of $V(\Gamma)$ into $n$ blocks each of which is a $(0,k)$-regular set in $\Gamma$.
\end{proof}

In~\cite[Theorem 2]{Lee01}, Lee gave a necessary and sufficient condition for the existence of a $(0, 1)$-regular set in Cayley graphs of abelian groups. This result can be generalized as follows.

\begin{thm}
\label{th:th2.0}
Let $G$ be a group, $H$ a proper subgroup of $G$, $S$ an inverse-closed subset of $G \setminus \{e\}$ such that $S \subset N_{G}(H)$, and $k$ a positive integer. Then $H$ is a $(0,k)$-regular set in $\cay(G,S)$ if and only if there exists a $k$-covering $f: \cay(G,S) \rightarrow K_{n}$ such that $H$ is a vertex-fiber of $f$, where $n = |G:H|$.
\end{thm}

\begin{proof}
Denote $n = |G:H|$. If $H$ is a $(0,k)$-regular set of $\cay(G,S)$, then by Lemma \ref{le:le2.0} and our assumption $S \subset N_{G}(H)$, $sH=Hs$ is also a $(0, k)$-regular set of $\cay(G,S)$ for any $s\in S$. Moreover, by Lemmas \ref{le:le2.2} and \ref{le:le1.4}, there exists an $|H|$-fold $k$-covering $f: \cay(G,S) \rightarrow K_{n}$ with vertex-fibers $gH$, $g \in L \cup \{e\}$, where $L$ is as in Lemma~\ref{le:le2.2}. In particular, $H$ is a vertex-fiber of $f$.

Conversely, suppose that there exists a $k$-covering $f: \cay(G,S) \rightarrow K_{n}$ such that $H$ is a vertex-fiber of $f$. Then all elements of $H$ are mapped by $f$ to the same vertex of $K_n$, say, $x$, and moreover $H = f^{-1}(x)$. Since $\{x\}$ is a $(0, 1)$-regular set in $K_n$, by part (a) of Lemma \ref{le:le1.4}, $H$ is a $(0,k)$-regular set in $\cay(G,S)$.
\end{proof}

The next lemma will play an important role in our subsequent discussions. In the special case when $k=1$, the equivalence of (\ref{it:ita}), (\ref{it:itd}) and (\ref{it:ite}) in this lemma gives rise to \cite[Lemma 2.2]{Ma19} which in turn is an extension of \cite[Lemma 2.1(a)]{Huang18}.

\begin{lem}
\label{co:co2.1}
Let $G$ be a group, $H$ a proper subgroup of $G$, $S$ an inverse-closed subset of $G \setminus \{e\}$, and $k$ a positive integer. Then the following conditions are equivalent:
\begin{enumerate}[\rm (a)]
\item\label{it:ita}
$H$ is a $(0,k)$-regular set of $\cay(G,S)$;
\item\label{it:itb}
there exists a subset $L$ of $S$ such that $L \cup \{e\}$ is a left transversal of $H$ in $G$ and $|\{g \in S \setminus L: gH = sH\}| = k-1$ for each $s \in L$;
\item\label{it:itc}
$H\cap S=\emptyset$ and $S$ can be partitioned into $n - 1$ parts each with size $k$ such that $sH = s'H$ for $s, s'\in S$ in the same part and $sH \ne s'H$ for $s, s'\in S$ in different parts;
\item\label{it:itd}
$S$ can be partitioned as $\{L_1,L_2,\ldots,L_k\}$ such that $L_i \cup \{e\}$ is a left transversal of $H$ in $G$ for each $1\leq i\leq k$;
\item\label{it:ite}
$S$ can be partitioned as $\{R_1,R_2,\ldots,R_k\}$ such that $R_i \cup \{e\}$ is a right transversal of $H$ in $G$ for each $1\leq i\leq k$.
\end{enumerate}
\end{lem}

\begin{proof}
Denote $\Ga = \cay(G,S)$ and $n = |G:H|$.

(\ref{it:ita}) $\Rightarrow$ (\ref{it:itb})\,
Suppose that $H$ is a $(0,k)$-regular set in $\Gamma$. By Lemma~\ref{le:le2.2}, there exists a subset $L$ of $S$ such that $L \cup \{e\}$ is a left transversal of $H$ in $G$. Note that $L \subseteq G \setminus H$ and each $s \in L$ is adjacent to $e \in H$ in $\Ga$. Since $H$ is a $(0,k)$-regular set in $\Gamma$, there exist exactly $k-1$ distinct elements $h_1, h_2, \ldots, h_{k-1} \in H \setminus \{e\}$ which are adjacent to $s$ in $\Ga$. That is, there exist exactly $k-1$ distinct elements $s_1,s_2,\ldots,s_{k-1} \in S$ such that $s = s_1h_1 = s_2h_2 =\cdots=s_{k-1}h_{k-1}$. Note that $s \ne s_i$ for each $i$ as $h_i \ne e$. If $s_i \in L$ for some $i$, then $sH = (s_i h_i)H = s_i H$. Since both $s$ and $s_i$ are in the left transversal $L \cup \{e\}$ of $H$ in $G$, it follows that $s = s_i$ and hence $h_i = e$, which is a contradiction. Hence $s_i \notin L$ for each $1 \le i \le k-1$. Thus, we have proved that for each $s \in L$ there are exactly $k-1$ distinct elements $s_{1}, s_2, \ldots,s_{k-1} \in S\setminus L$ such that $s_{1}H = s_{2}H = \cdots = s_{k-1}H = sH$.

(\ref{it:itb}) $\Rightarrow$ (\ref{it:itc})\,
Suppose that $L$ is a subset of $S$ which satisfies condition (b). Since $L\cup \{e\}$ is a left transversal of $H$ in $G$, we have $|L|=n-1$, $H \cap S = \emptyset$, and $sH\neq s'H$ for any distinct elements $s, s'$ of $L$. By our assumption, for each $s \in L$, the subset $T_s = \{g \in S: gH = sH\}$ of $S$ has size $k$. Moreover, we have $T_s \cap T_{s'} = \emptyset$ for distinct $s, s' \in L$ and $\cup_{s \in L} T_s = S$. Thus, $\{T_s: s\in L\}$ is a partition of $S$ into $n-1$ parts such that $sH = s'H$ for $s, s'$ in the same part and $sH \ne s'H$ for $s, s'$ in different parts.

(\ref{it:itc}) $\Rightarrow$ (\ref{it:itd})\,
Suppose that $\{T_1, T_2, \ldots, T_{n-1}\}$ is a partition of $S$ satisfying condition (c). Then $|T_j|=k$ for $1 \leq j \leq n-1$. Write $T_j = \{s_{1,j}, s_{2,j}, \ldots, s_{k,j}\}$ for $1 \leq j \leq n-1$. By our assumption, we have $s_{1,j}H = s_{2,j}H = \cdots =  s_{k,j}H$ for $1 \leq j \leq n-1$. Set $L_i = \{s_{i,1}, s_{i,2}, \ldots, s_{i,n-1}\}$ for $1 \le i \le k$. Then $L_1, L_2, \ldots, L_k$ form a partition of $S$. Moreover, by \eqref{le:le1.1} and our assumption, $L_i\cup\{e\}$ is a left transversal of $H$ in $G$ for each $1\leq i\le k$.

(\ref{it:itd}) $\Rightarrow$ (\ref{it:ita})\,
Suppose that condition (d) holds. Since each $L_i \cup \{e\}$ is a left transversal of $H$ in $G$, we have $|L_i| = n - 1$, $|S| = k(n-1)$,  and for any $u \in G \setminus H$ there exists exactly one element $s_{i(u)} \in L_i$ such that $u \in s_{i(u)}H$, say, $u = s_{i(u)}h_{i(u)}$ for some $h_{i(u)} \in H$. Since $\{L_1,L_2,\ldots,L_k\}$ is a partition of $S$, we have $s_{i(u)} \ne s_{j(u)}$ whenever $i \neq j$ and hence $h_{1(u)}, h_{2(u)}, \ldots, h_{k(u)}$ are distinct neighbours of $u$ in $H$. Since $\{L_1,L_2,\ldots,L_k\}$ is a partition of $S$ and each $L_i \cup \{e\}$ is a left transversal of $H$ in $G$, these are all neighbours of $u$ in $H$ and moreover $H \cap S = \emptyset$. Hence $H$ is an independent set of $\Ga$. Thus, any $u \in G \setminus H$ is adjacent to exactly $k$ distinct vertices in $H$ and therefore $H$ is a $(0,k)$-regular set in $\Gamma$.

(\ref{it:itd}) $\Rightarrow$ (\ref{it:ite})\,
Suppose that condition (\ref{it:itd}) holds. Since $\{L_1,L_2,\ldots,L_k\}$ is a partition of $S$ and $L_i \cup\{e\}$, $1\leq i\leq k$, are left transversals of $H$ in $G$, the elements of $S$ can be listed as $S = \{s_{i,j}: 1\leq i\leq k, 1\leq j\leq n-1\}$ such that $L_i = \{s_{i,1}, s_{i,2}, \ldots, s_{i,n-1}\}$ for $1 \le i \le k$ and
$s_{1,j}H=s_{2,j}H=\cdots=s_{k,j}H$ for $1 \le j \le n-1$. Thus $Hs^{-1}_{1,j} = Hs^{-1}_{2,j} = \cdots = Hs^{-1}_{k,j}$. Since $S$ is inverse-closed, we have $\{s^{-1}_{1,j}, s^{-1}_{2,j}, \ldots, s^{-1}_{k,j}\} \subseteq S$ for $1 \le j \le n-1$. Set $R_i = \{s^{-1}_{i,1}, s^{-1}_{i,2}, \ldots, s^{-1}_{i,n-1}\}$ for $1\leq i\leq k$. Then $\{R_1,R_2,\ldots,R_k\}$ is a partition of $S$. Since each $L_i \cup\{e\}$ is a left transversal of $H$ in $G$, we see that $H, Hs^{-1}_{i,1}, Hs^{-1}_{i,2}, \ldots, Hs^{-1}_{i, n-1}$ are pairwise distinct. This together with $n = |G:H|$ implies that $R_i \cup \{e\}$ is a right transversal of $H$ in $G$ for $1\leq i\leq k$.

(\ref{it:ite}) $\Rightarrow$ (\ref{it:itd})\,
This is similar to the proof that (\ref{it:itd}) implies (\ref{it:ite}).
\end{proof}

\section{$(0, k)$-regular sets and equitable partitions}
\label{sec:eq-partn}

Denote by $A(\Gamma)$ the adjacency matrix of a graph $\Gamma$. A partition $\pi=\{V_1, V_2, \ldots, V_r\}$ of $V(\Gamma)$ is called an {\em equitable partition} of $\Ga$ (see, for example, \cite{god01}) if there is an $r\times r$ matrix $M = (m_{ij})$, called the {\em quotient matrix} of $\pi$, such that for any $1\leq i, j\leq r$ every vertex in $V_i$ has exactly $m_{ij}$ neighbours in $V_j$.
It is known that, for an equitable partition $\pi$ with quotient matrix $M$, the characteristic polynomial of $M$ divides that of $A(\Gamma)$ (see, for example, \cite{god01}), or, equivalently, the spectrum of $M$ is contained in the spectrum of $A(\Gamma)$ (see, for example, \cite{Bai19}). In the case when $\Gamma$ is a connected $d$-regular graph, the equitable partition $\pi$ is said to be {\em $\mu$-equitable} if all eigenvalues of $M$ other than its principal eigenvalue $d$ are equal to $\mu$, and a non-empty proper set $D$ of $V(\Gamma)$ is called a {\em $\mu$-perfect set} of $\Gamma$ if $\{D, V(\Gamma)\setminus D\}$ is an $\mu$-equitable partition of $\Gamma$ .

\begin{lem}
\cite[Proposition 2.1]{Bai19}
\label{le:le4.12}
Let $\Gamma$ be a connected regular graph, and let $\pi=\{V_1, V_2, \ldots,V_r\}$ be a partition of $V(\Gamma)$. If $\pi$ is $\mu$-equitable, then each $V_i$ is a $\mu$-perfect set of $\Gamma$. Conversely, if $V_1, V_2, \ldots, V_r$ are all $\mu$-perfect sets of $\Gamma$, then $\pi$ is $\mu$-equitable.
\end{lem}

\begin{lem}
\label{le:le4.13}
Let $\Gamma$ be a connected regular graph, and let $\pi=\{V_1, V_2, \ldots, V_r\}$ be a $(-k)$-equitable partition of $V(\Gamma)$. If for some $i$, $V_i$ is an independent set of $\Gamma$, then $k$ is a positive integer and $V_i$ is a $(0,k)$-regular set in $\Gamma$.
\end{lem}

\begin{proof}
Denote by $d$ the degree of $\Ga$. We may assume without loss of generality that $V_1$ is an independent set of $\Gamma$. Since $\pi$ is $(-k)$-equitable, by Lemma~\ref{le:le4.12}, $\pi' = \{V_1, V(\Gamma) \setminus V_1\}$ is a $(-k)$-equitable partition of $V(\Gamma)$. Since $V_1$ is an independent set of $\Ga$, the quotient matrix of $\pi'$ is of the form
$$
M' = \pmat{0 & a_{12}\\ d & d-a_{12}}
$$
for some integer $a_{12}$. Note that $M'$ has eigenvalues $d$ and $-a_{12}$. On the other hand, by the definition of a $(-k)$-equitable partition, both $d$ and $-k$ must be eigenvalues of $M'$. Hence $a_{12} = k$. In other words, any vertex in $V(\Gamma)\setminus V_1$ has exactly $k$ neighbours in $V_1$ and therefore $V_1$ is a $(0,k)$-regular set in $\Gamma$.
\end{proof}

The main result in this section is as follows.

\begin{thm}
\label{th:th4.11}
Let $G$ be a group, $H$ a proper subgroup of $G$, $S$ an inverse-closed subset of $G \setminus \{e\}$ such that $S \subset N_{G}(H)$, and $k$ a positive integer. Then the following conditions are equivalent:
\begin{enumerate}[\rm (a)]
\item $H$ is a $(0,k)$-regular set in $\cay(G,S)$;
\item $H\cap S = \emptyset$ and $\cay(G,S)$ has a $(-k)$-equitable partition with exactly $|G:H|$ blocks and with $H$ as one of the blocks;
\item for every integer $a$ with $0 \le a \leq \frac{|S|}{k} - 1$, $\cay(G,S)$ admits an $(ak,(a+1)k)$-regular set which contains $H$ as a subset.
\end{enumerate}
\end{thm}

\begin{proof}
Denote $\Ga = \cay(G,S)$ and $n = |G:H|$.

(b) $\Rightarrow$ (a)\, Suppose that condition (b) holds. Since $H$ is a subgroup of $G$, the assumption $H\cap S=\emptyset$ implies that $H$ is an independent set of $\Ga$. So, by Lemma \ref{le:le4.13}, $H$ is a $(0,k)$-regular set in $\Ga$.

(a) $\Rightarrow$ (b)\, Suppose that $H$ is a $(0,k)$-regular set in $\Ga$. Then $H$ is an independent set of $\Ga$ and hence $H\cap S=\emptyset$. By Lemma \ref{le:le2.2}, there exists a subset $L$ of $S$ with size $|L|=|G:H|-1$ such that $L\cup \{e\}$ is a left transversal of $H$ in $G$.
Let $L=\{s_1, s_2,\ldots, s_{n-1}\}$. Since $S\subset N_{G}(H)$, by Lemma \ref{le:le2.0}, $s_iH = Hs_{i}$ is a $(0, k)$-regular set in $\Ga$ for $1 \le i \le n-1$. Using this and a straightforward computation, we obtain that for $1 \le i \le n-1$, the quotient matrix of the partition $\{s_iH, G \setminus s_iH\}$ has exactly two eigenvalues, namely $-k$ and $|S|$. Thus, $\{s_iH, G \setminus s_iH\}$  is a $(-k)$-equitable partition of $\Ga$ for $1 \le i \le n-1$, and therefore $H, s_1H, s_2H,\ldots, s_{n-1}H$ are all $(-k)$-perfect sets of $\Ga$. It follows from Lemma \ref{le:le4.12} that $\{H, s_1H, s_2H,\ldots, s_{n-1}H\}$ is a $(-k)$-equitable partition of $\Ga$ with exactly $|G:H|$ blocks and with $H$ as one of the blocks.

(c) $\Rightarrow$ (a)\, Setting $a=0$ in (c), we obtain (a).

(a) $\Rightarrow$ (c)\, Suppose that $H$ is a $(0,k)$-regular set in $\cay(G,S)$. Then by what we have proved above, condition (b) is satisfied. By Lemma \ref{le:le2.2}, there exists a subset $L$ of $S$ with $|L|=|G:H|-1$ such that $L\cup \{e\}$ is a left transversal of $H$ in $G$. Let $L=\{s_1, s_2,\ldots, s_{n-1}\}$.
By the statement in (b), $\pi = \{H, s_1H, s_2H, \ldots, s_{n-1}H\}$ is a $(-k)$-equitable partition of $\cay(G,S)$ containing $H$ as a block. For any integer $a$ with $0 \le a \leq \frac{|S|}{k} - 1$, let $H_a=H\cup s_1H\cup\cdots\cup s_aH$ and $\pi_a = \{H_a, s_{a+1}H, \ldots, s_{n-1}H\}$. Set $s_0 = e$. Since $H$ is a $(0,k)$-regular set in $\cay(G,S)$ and $S \subset N_{G}(H)$, by Lemma \ref{le:le2.0}, $s_iH = Hs_i$ is a $(0, k)$-regular set in $\cay(G,S)$ for $0 \le i \le n-1$. Thus, for $0 \leq i, j \leq a$ with $i \ne j$, each $u \in s_iH$ has exactly $k$ neighbours in $s_jH$. Hence each $u \in H_a$ has exactly $ak$ neighbours in $H_a$. On the other hand, each $u \in G \setminus H_a$ has exactly $k$ neighbours in $s_iH$, for $0\leq i\leq a$. So each $u \in G \setminus H_a$ has exactly $(a+1)k$ neighbours in $H_a$. Therefore, $H_a$ is an $(ak, a(k+1))$-regular set in $\cay(G,S)$.
Obviously, it contains $H$ as a subset.
\end{proof}

In general, the three conditions in Theorem \ref{th:th4.11} may not be equivalent if $H$ is a subset of $G$ satisfying $sH = Hs$ for any $s\in S$ but not a subgroup of $G$. However, we have the following proposition which shows that (a) and (c) in Theorem \ref{th:th4.11} are equivalent when $H$ is such a subset and $k=1$.

\begin{pro}
\label{co:co6.2}
Let $G$ be a group, $D$ a proper subset of $G$, and $S$ an inverse-closed subset of $G \setminus \{e\}$ such that $sD=Ds$ for any $s\in S$. Then $\cay(G,S)$ admits $D$ as a $(0,1)$-regular set if and only if, for every integer $a$ with $0 \leq a \leq |S| - 1$, $\cay(G,S)$ admits an $(a, a+1)$-regular set which contains $D$ as a subset.
\end{pro}

\begin{proof}
The sufficiency is trivially true. To prove the necessity, suppose that $D$ is a $(0, 1)$-regular set in $\cay(G, S)$. Assume that $S=\{s_1, s_2, \ldots, s_n\}$.  For any integer $a$ with $0\leq a\leq |S|-1$, let $D_a = s_0D \cup s_1D \cup \cdots \cup s_aD$, where $s_0 = e$. Since $D$ is a $(0,1)$-regular set in $\cay(G,S)$ and $s_i D = Ds_i$, by \cite[Lemma 3]{Lee01}, $s_i D$ is a $(0, 1)$-regular set in $\cay(G,S)$ for $0 \le i \le n$ and $\{s_0 D, s_1 D, s_2 D, \ldots, s_n D\}$ forms a partition of $G$. Thus, for $0 \leq i, j \leq a$ with $i \ne j$, each $u \in s_i D$ has exactly one neighbour in $s_j D$. Hence each $u \in D_a$ has exactly $a$ neighbours in $D_a$. On the other hand, each $u \in G \setminus D_a$ has exactly one neighbour in $s_i D$, for $0\leq i\leq a$. So each $u \in G \setminus D_a$ has exactly $a+1$ neighbours in $D_a$. Therefore, $D_a$ is an $(a, a+1)$-regular set in $\cay(G,S)$. Obviously, it contains $D$ as a subset.
\end{proof}

A key ingredient in the proof above is the fact that $\{s_0 D, s_1 D, s_2 D, \ldots, s_n D\}$ forms a partition of $G$. The counterpart of this statement may not be true when $k \ge 2$, and this explains why we require $k=1$ in Proposition \ref{co:co6.2}.

\section{Normal subgroups as $(0,k)$-regular sets}
\label{sec:secnor}

In this section we study when a proper normal subgroup of a group is a $(0,k)$-regular set of the group. The first main result in this section is as follows.

\begin{thm}
\label{th:th5.3}
Let $G$ be a group and $H$ a non-trivial proper normal subgroup of $G$. Then for any even integer $k$ between $2$ and $|H|$, $H$ is a $(0, k)$-regular set of $G$.
\end{thm}

\begin{proof}
Set $n = |G:H|$. We first prove the result in the case when $k = 2$. Take a subset $R = \{r_1, \ldots, r_{n-1}\}$ of $G \setminus \{e\}$ such that $R \cup \{e\}$ is a right transversal of $H$ in $G$. Since $H$ is normal in $G$, $R^{-1} \cup \{e\}$ is also a right transversal of $H$ in $G$. If $R \cap R^{-1} = \emptyset$, then by Lemma \ref{co:co2.1}, $H$ is a $(0,2)$-regular set in $\cay(G, R \cup R^{-1})$ and hence is a $(0,2)$-regular set of $G$.

Assume $|R\cap R^{-1}| \geq1$ from now on. Without loss of generality we may assume that $R\cap R^{-1} = \{r_1, r_2,\ldots, r_{m}\}$, where $1 \le m \le n-1$. By Lemma \ref{co:co2.1}, it suffices to construct two disjoint subsets $R_{1}, R_{2}$ of $G$ such that $R_{1}\cup R_{2} = (R_{1}\cup R_{2})^{-1}$ and each of $R_{1}\cup\{e\}$ and $R_{2}\cup \{e\}$ is a right transversal of $H$ in $G$. We deal with the following two cases separately.

\smallskip
\textsf{Case 1.} $|H|=2$ and $H=\{e, h\}$.
\smallskip

Define $Q$ to be the subset of $G$ obtained from $R^{-1}$ by replacing each involution $r_i \in R\cap R^{-1}$ by $hr_i$ and retaining all other elements of $R^{-1}$. Since $Hr_i = Hhr_i$ for each involution $r_i \in R\cap R^{-1}$, $Q \cup \{e\}$ is a right transversal of $H$ in $G$. Since both $r_i$ and $h$ are involutions and $H$ is normal in $G$, we have $r_{i}^{-1} h^{-1} = r_{i}h = h r_{i}$ and hence $(R\cup Q)^{-1}=R\cup Q$.

If $R\cap Q=\emptyset$, then by Lemma \ref{co:co2.1}, $H$ is a $(0,2)$-regular set in $\cay(G, R \cup Q)$ and hence $H$ is a $(0,2)$-regular set of $G$.

Assume $R\cap Q\neq\emptyset$ in the sequel. Without loss of generality we may assume that $R\cap Q=\{r_1, r_2, \ldots, r_t, r_1^{-1},\ldots, r_t^{-1}\}\subseteq R\cap R^{-1}$, where $1 \leq t \leq \frac{m}{2}$. Since both $R \cup \{e\}$ and $R^{-1} \cup \{e\}$ are right transversals of $H$ in $G$, we have $Hr_i^{-1}\neq Hr_i$ for each $r_i \in R\cap Q$. Let $R_{1}$ be obtained from $Q$ by replacing $r_i, r_i^{-1}$ by $hr_i, hr_i^{-1}$, respectively, for $1 \le i \le t$, but retaining all other elements. Since $Hr_i=Hhr_i$ and $Hr_i^{-1}=Hhr_i^{-1}$, $R_{1} \cup \{e\}$ is a right transversal of $H$ in $G$. Set $R_{2}=R$. Since $h^2=e$, we have $(R_1 \cup R_{2})^{-1} = R_{1} \cup R_{2}$. Moreover, $R_{1}\cap R_{2}=\emptyset$ and each of $R_{1}\cup\{e\}$ and $R_{2}\cup \{e\}$ is a right transversal of $H$ in $G$. Thus, by Lemma \ref{co:co2.1}, $H$ is a $(0,2)$-regular set in $\cay(G, R_{1} \cup R_{2})$ and hence a $(0,2)$-regular set of $G$.

\smallskip
\textsf{Case 2.} $|H|\geq 3$.
\smallskip

For all involutions $r_i \in R\cap R^{-1}$ such that $Hr_i$ contains an involution, say, $hr_i$, replace $r_i$ in $R^{-1}$ by $hr_i$ and retain all other elements of $R^{-1}$. Denote by $Q$ the subset of $G$ obtained this way.
If $r_i \in R\cap R^{-1}$ is an involution such that $Hr_i$ has no involutions, then there exist distinct elements $hr_i, h'r_i$ of $Hr_i$ such that $(hr_i)^{-1}=h'r_i$ as $H$ is normal in $G$. For all such involutions $r_i$, replace $r_i$ in $R$ by $hr_i$ to obtain $Q_0$ and replace $r_i^{-1}$ in $Q$ by $h'r_i$ to obtain $Q_1$.
Since $H(hr_i)=Hr_i=H(h'r_i)$, both $Q_0 \cup \{e\}$ and $Q_1 \cup \{e\}$ are transversals of $H$ in $G$. Moreover, $(Q_0\cup Q_1)^{-1} = Q_0\cup Q_1$.

If $Q_0\cap Q_1=\emptyset$, then by Lemma \ref{co:co2.1}, $H$ is a $(0,2)$-regular set in $\cay(G, Q_0 \cup Q_1)$ and hence a $(0,2)$-regular set of $G$.

Assume $Q_0\cap Q_1\neq\emptyset$ in the sequel. Without loss of generality we may assume that $Q_0\cap Q_1=\{r_1, r_2, \ldots, r_t, r_1^{-1},r_2^{-1},\ldots, r_t^{-1}\}\subseteq R\cap R^{-1}$, where $1 \leq t \leq \frac{m}{2}$. Since both $R \cup \{e\}$ and $R^{-1} \cup \{e\}$ are right transversals of $H$ in $G$, we have $Hr_i^{-1}\neq Hr_i$ for each $r_i \in Q_0\cap Q_1$. Since $H$ is normal in $G$ and $r_i\neq r_i^{-1}$, there exist elements $h, h'$ of $H$ such that $hr_i \neq (hr_i)^{-1} = h'r_{i}^{-1}$. Define $R_{1}$ to be the subset of $G$ obtained from $Q_1$ by replacing $r_i, r_i^{-1}$ by $hr_i, h'r_i^{-1}$, respectively, for $1 \le i \le t$. Since $Hr_i = Hhr_i$ and $Hr_i^{-1} = Hhr_i^{-1}$, $R_{1} \cup \{e\}$ is a right transversal of $H$ in $G$. Moreover, $(Q_0\cup R_{1})^{-1}=Q_0\cup R_{1}$. Set $R_{2} = Q_0$. Then $R_{1}\cap R_{2} = \emptyset$, $(R_{1}\cup R_{2})^{-1} = R_{1}\cup R_{2}$, and both $R_{1}\cup\{e\}$ and $R_{2}\cup \{e\}$ are right transversals of $H$ in $G$. Therefore, by Lemma \ref{co:co2.1}, $H$ is a $(0,2)$-regular set in $\cay(G, R_{1} \cup R_{2})$ and hence a $(0,2)$-regular set of $G$.

Up to now we have proved the result in the case when $k = 2$. With this as the base case the rest of the proof proceeds by induction on $k$. Suppose inductively that for some $l$ with $1 \le l \le \lfloor |H|/2 \rfloor-1$
we have obtained $2l$ pairwise disjoint subsets $R_1, R_2, \ldots, R_{2l}$ of $G$ such that $R_i \cup \{e\}$ is a right transversal of $H$ in $G$ for $1\leq i\leq 2l$ and $\cup_{j=1}^{2l} R_j$ is inverse-closed. Write $R_1=\{r_{1,1}, r_{1,2},\ldots, r_{1,n-1}\}$. Since $2l < |H|$, for each $1 \le i \le n-1$, there exists $h_{i} \in H$ such that $h_{i}r_{1,i} \notin \cup_{j=1}^{2l} R_j$. Set $T=\{h_1r_{1,1}, h_{2}r_{1,2}, \ldots h_{n-1}r_{1,n-1}\}$.
Since $R_1\cup\{e\}$ is a right transversal of $H$ in $G$ and $Hh_{i}r_{1,i}=Hr_{1,i}$ for each $h_{i}r_{1,i} \in T$, $T\cup \{e\}$ is a right transversal of $H$ in $G$. Since $H$ is normal in $G$, $T^{-1}\cup \{e\}$ is also a right transversal of $H$ in $G$. Note that $(\cup_{i=1}^{2l}R_{i})^{-1}=\cup_{i=1}^{2l}R_{i}$ and $T\cap R_{i}=\emptyset$ for $1\leq i\leq 2l$. Hence $T^{-1}\cap R_{i}=\emptyset$ for $1\leq i\leq 2l$.

If $T\cap T^{-1}=\emptyset$, then by Lemma \ref{co:co2.1},  $H$ is a $(0, 2l+2)$-regular set of $\cay(G, (\cup_{i=1}^{2l}R_{i})\cup T\cup T^{-1})$ and hence a $(0, 2l+2)$-regular set of $G$.

In what follows we assume $T\cap T^{-1}\neq\emptyset$, say, $T\cap T^{-1} = \{t_1, t_2,\ldots, t_{m}\}$, where $1 \le m \le n-1$. If $|H|=2$, then $H$ is a $(0, 2)$-regular set of $G$ as shown in the case when $k=2$. Henceforth we assume further that $|H|\geq 3$. For all involutions $t_i \in T\cap T^{-1}$ such that $Ht_i \setminus (\cup_{j=1}^{2l}R_j)$ contains an involution, say, $ht_i$, replace $t_i$ by $ht_i$ in $T^{-1}$ to obtain a new subset $T_1$ of $G$. Since $H$ is normal in $G$ with $|H| \ge 3$, if $t_i \in T\cap T^{-1}$ is an involution such that $Ht_i \setminus (\cup_{j=1}^{2l}R_j)$ has no involutions, then there exist distinct elements $ht_i, h't_i$ of $Ht_i \setminus (\cup_{j=1}^{2l}R_j)$ such that $(ht_i)^{-1} = h't_i$. For all such involutions $t_i$, replace $t_i$ by $ht_i$ in $T$ to obtain a new subset $T_0$ and replace $t_i^{-1}$ by $h't_i$ in $T_1$ to obtain a new subset $T'_1$. Since $H(ht_i)=Ht_i=H(h't_i)$, both $T_0 \cup \{e\}$ and $T'_1 \cup \{e\}$ are transversals of $H$ in $G$. Moreover, $(T_0\cup T'_1)^{-1} = T_0\cup T'_1$.

If $T_0\cap T'_1=\emptyset$, then by Lemma \ref{co:co2.1}, $H$ is a $(0,2l+2)$-regular set in $\cay(G, (\cup_{i=1}^{2l}R_{i})\cup T_0\cup T'_1)$ and hence a $(0,2l+2)$-regular set of $G$.

Assume $T_0\cap T'_1\neq\emptyset$ in the sequel. Without loss of generality we may assume that $T_0\cap T'_1=\{t_{1}, t_{2}, \ldots, t_r, t^{-1}_{1}, t^{-1}_{2}, \ldots, t_r^{-1}\}\subset T\cap T^{-1}$, where $1\leq r\leq \frac{m}{2}$. Since both $T \cup \{e\}$ and $T^{-1} \cup \{e\}$ are right transversals of $H$ in $G$, we have $Ht_i^{-1}\neq Ht_i$ for $1 \le i \le r$. Since $H$ is normal in $G$ and $t_i\neq t_i^{-1}$, there exist elements $h, h'$ of $H$ such that $ht_i \neq (ht_i)^{-1} = h't_{i}^{-1}$. Let  $T''$ be the subset of $G$ obtained from $T'_1$ by replacing $t_i, t_i^{-1}$ by $ht_i, h't_i^{-1}$, respectively, for $1 \le i \le r$. Since $Ht_i=Hht_i$ and $Ht_i^{-1}=Hht_i^{-1}$, $T'' \cup \{e\}$ is a right transversal of $H$ in $G$. Also, we have $(T_0\cup T'')^{-1} = T_0 \cup T''$. Note that $T_0 \cap T'' = \emptyset$ and both $T_0 \cup\{e\}$ and $T'' \cup \{e\}$ are right transversals of $H$ in $G$. Therefore, by Lemma \ref{co:co2.1}, $H$ is a $(0,2l+2)$-regular set in $\cay(G, (\cup_{i=1}^{2l}R_{i})\cup T_0 \cup T'')$ and hence a $(0,2l+2)$-regular set of $G$. This completes the proof.
\end{proof}

The following result was proved by Huang, Xia and Zhou in~\cite{Huang18}.

\begin{lem}~\cite[Theorem 2.2]{Huang18}
\label{le:lehuang}
Let $G$ be a group and $H$ a proper normal subgroup of $G$. Then $H$ is a $(0,1)$-regular set of $G$ if and only if for any $g \in G$ with $g^2 \in H$ there exists an element $h \in H$ such that $(gh)^2=e$.
\end{lem}

This result was generalized as follows by Wang, Xia and Zhou in \cite{Wang20}.

\begin{lem}~\cite[Theorem 1.2]{Wang20}
\label{le:lewang20}
Let $G$ be a group and $H$ a non-trivial proper normal subgroup of $G$. Then $H$ is a $(0,1)$-regular set of $G$ if and only if $H$ is an $(a, b)$-regular set of $G$ for every pair of integers $a, b$ with $0\leq a \leq |H|- 1$ and $0 \leq b \leq |H|$ such that $\gcd(2, |H|-1)$ divides $a$.
\end{lem}

Using Lemmas \ref{le:lehuang} and \ref{le:lewang20}, we obtain the following result.

\begin{thm}
\label{th:th5.4}
Let $G$ be a group and $H$ a non-trivial proper normal subgroup of $G$. Then for any odd integer $k$ between $1$ and $|H|$, $H$ is a $(0, k)$-regular set of $G$ if and only if for any $g \in G$ with $g^2 \in H$ there exists an element $h \in H$ such that $(gh)^2=e$.
\end{thm}

\begin{proof}
The sufficiency follows Lemmas~\ref{le:lehuang} and~\ref{le:lewang20}. To prove the necessity, suppose that $k$ is an odd integer between $1$ and $|H|$ and $H$ is a $(0, k)$-regular set of $G$. Then by Lemma~\ref{co:co2.1} there exists an inverse-closed subset $S$ of $G$ which can be partitioned into $L_1, L_2, \ldots, L_{k}$ such that $L_i \cup\{e\}$ is a left transversal of $H$ in $G$ for $1\leq
i\leq k$. Note that $|S| = k |L_i| = k(n-1)$ for each $i$, where $n = |G:H|$. So we can write $S = \{s_{i,j}: 1 \le i \le k, 1 \le j \le n-1\}$ such that $L_i = \{s_{i,1}, s_{i,2}, \ldots, s_{i,n-1}\}$ for $1 \le i \le k$ and $s_{1,j}H = s_{2,j}H = \cdots = s_{k,j}H$ for $1 \le j \le n-1$. Consider any $g \in G \setminus H$ with $g^2 \in H$. For each $1 \leq i \leq k$, there exists an element $s_{i, j_i} \in L_i$ such that $g \in s_{i, j_i}H$, where $j_i$ depends on $i$. Since $g^2 \in H$ and $H$ is normal in $G$, we have $s_{i,j_i}H = gH = g^{-1}H = Hg^{-1} = s_{i,j_i}^{-1}H = Hs_{i,j_i}^{-1}$. Set $T = \{s_{i, j_i}\in S: gH=s_{i,j_i}H\}$. Then $T^{-1} = T$. Since $|T|=k$ and $k$ is odd, it follows that there exists $1 \le i \le k$ such that $s_{i, j_i}^{-1} = s_{i, j_i}$. Since $g \in s_{i, j_i}H$, we have $s_{i, j_i} = gh$ for some $h \in H$. Hence $(gh)^2 = s_{i, j_i}^{2} = e$. So we have proved that for any $g \in G \setminus H$ with $g^2 \in H$ there exists an element $h \in H$ such that $(gh)^2=e$. Obviously, the same statement holds for any $g \in H$ as well. This completes the proof.
\end{proof}

Obviously, Theorem~\ref{th:th5.4} is a generalization of Lemma~\ref{le:lehuang}, but on the other hand its proof relies on Lemma~\ref{le:lehuang}. Note that the ``only if'' part in Theorem~\ref{th:th5.4} is not implied in Lemmas ~\ref{le:lehuang} and \ref{le:lewang20} as our $k$ is fixed. Theorem~\ref{th:th5.4} implies the following corollary.

\begin{cor}
\label{co:co5.0}
Let $G$ be a group and $H$ a non-trivial proper normal subgroup of $G$. Then for any odd integers $k, l$ between $1$ and $|H|$, $H$ is a $(0, k)$-regular set of $G$ if and only if it is a $(0, l)$-regular set of $G$. In particular, for any odd integer $k$ between $1$ and $|H|$, $H$ is a $(0, 1)$-regular set of $G$ if and only if it is a $(0, k)$-regular set of $G$.
\end{cor}

Note that the ``only if" part of the second statement in this corollary is implied in the ``only if" part of Lemma \ref{le:lewang20}, but the ``if" part of this statement is stronger than the ``if" part of Lemma \ref{le:lewang20} as $k$ is fixed.

Corollary \ref{co:co5.0} and part (a) of Corollaries 2.3 and 2.4 in \cite{Huang18} together imply the following result.

\begin{cor}
\label{co:co5.2}
Let $G$ be a group and $H$ a non-trivial proper normal subgroup of $G$. If either $|H|$ or $|G:H|$ is odd, then for any odd integer $k$ between $1$ and $|H|$, $H$ is a $(0, k)$-regular set of $G$. In particular, if $G$ is of odd order, then any non-trivial normal subgroup $H$ of $G$ is a $(0, k)$-regular set of $G$ for any odd integer $k$ between $1$ and $|H|$.
\end{cor}

The following lemma consists of three results from \cite{Huang18}. We will use this lemma to prove our final result in this section.

\begin{lem}
\label{le:lehuang2}
Let $G$ be an abelian group with Sylow $2$-subgroup $P=\langle a_1\rangle\times\cdots\times\langle a_n\rangle,$ where $a_i$ has order $2^{m_i}>1$ for $1\leq i\leq n$. Let $H$ be a proper subgroup of $G$. Then the following statements hold:
\begin{enumerate}[\rm (a)]
\item\label{it:it2.6}
 $H$ is a $(0, 1)$-regular set of $G$ if and only if $H\cap P$ is a $(0,1)$-regular set of $P$ (\cite[Lemma 2.5]{Huang18});
\item\label{it:it2.7}
 if $H$ is a $(0, 1)$-regular set of $G$, then either $H\cap P$ is trivial or $H\cap P$ projects onto at least one of $\langle a_1\rangle,\ldots, \langle a_n\rangle$ (\cite[Lemma 2.6(a)]{Huang18});
\item\label{it:it2.8} if $H\cap P$ is cyclic, then $H$ is a $(0, 1)$-regular set of $G$ if and only if either $H\cap P$ is trivial or $H\cap P$ projects onto at least one of $\langle a_1\rangle,\ldots, \langle a_n\rangle$ (\cite[Theorem 2.7(a)]{Huang18}).
\end{enumerate}
\end{lem}

The three parts of the following result generalize \cite[Lemma 2.5]{Huang18}, \cite[Theorem 2.7(a)]{Huang18} and \cite[Corollary 2.8(a)]{Huang18}, respectively.

\begin{thm}
\label{le:le5.1}
Let $G$ be an abelian group with Sylow $2$-subgroup $P=\langle a_1\rangle\times\cdots\times\langle a_n\rangle$, where $a_i$ has order $2^{m_i}>1$ for $1\leq i\leq n$, and let $H$ be a proper subgroup of $G$. Then the following hold:
\begin{enumerate}[\rm (a)]
\item for any odd integer $k$ between $1$ and $|H|$, $H$ is a $(0, k)$-regular set of $G$ if and only if $H\cap P$ is a $(0, k)$-regular set of $P$;
\item if $H\cap P$ is cyclic, then for any odd integer $k$ between $1$ and $|H|$, $H$ is a $(0, k)$-regular set of $G$ if and only if either $H\cap P$ is trivial or $H\cap P$ projects onto at least one of $\langle a_1\rangle, \ldots, \langle a_n\rangle$;
\item if $G$ is a cyclic group, then for any odd integer $k$ between $1$ and $|H|$, $H$ is a $(0, k)$-regular set of $G$ if and only if either $|H|$ or $|G:H|$ is odd.
\end{enumerate}
\end{thm}

\begin{proof}
(a) In view of Corollary~\ref{co:co5.0}, for any odd integer $k$ between $1$ and $|H|$, $H$ is a $(0, k)$-regular set of $G$ if and only if it is a $(0,1)$-regular set of $G$. By part~(\ref{it:it2.6}) of Lemma~\ref{le:lehuang2}, $H$ is a $(0,1)$-regular set of $G$ if and only if $H\cap P$ is a $(0,1)$-regular set of $P$. By Corollary~\ref{co:co5.0} again, $H\cap P$ is a $(0,1)$-regular set of $P$ if and only if it is a $(0,k)$-regular set of $P$. Combining these statements, we obtain (a) immediately.

(b) By part (a), it suffices to prove (b) in the case when $G=P=\langle a_1\rangle\times\cdots\times\langle a_n\rangle$ and $H \cap P = H$ is cyclic. In this case, by Corollary \ref{co:co5.0}, $H$ is a $(0, k)$-regular set of $G$ for any odd integer $k$ between $1$ and $|H|$ if and only if it is a $(0, 1)$-regular set of $G$, which, by part (c) of Lemma~\ref{le:lehuang2}, is true if and only if either $H\cap P$ is trivial or $H$ projects onto at least one of $\langle a_1\rangle, \ldots, \langle a_n\rangle$.

(c) In the special case when $G$ is a cyclic group, a subgroup $H$ of $G$ is a $(0, k)$-regular set of $G$ for any odd integer $k$ between $1$ and $|H|$ if and only if either $H\cap P$ is trivial or $H \cap P = P$. However, $H\cap P$ is trivial if and only if $|H|$ is odd, whilst $H \cap P = P$ if and only if $|G:H|$ is odd. Therefore, $H$ is a $(0, k)$-regular set of $G$ for any odd integer $k$ between $1$ and $|H|$ if and only if either $|H|$ or $|G:H|$ is odd.
\end{proof}

\section{Subgroup $(0,k)$-regular sets of dihedral groups and generalized quaternion groups}
\label{sec:dih-qua}

Recall that the dihedral group $\dn$ of order $2n$ is defined as
\begin{equation}
\label{eq:D2n}
\dn=\langle a,b\ |\ a^n=b^2=e, bab=a^{-1}\rangle.
\end{equation}
It is known that the subgroups of $\dn$ are the cyclic groups $\langle a^t\rangle$ with $t$ dividing $n$ and the dihedral groups $\langle a^t, a^sb\rangle$ with $t$ dividing $n$ and $0\leq s\leq n-1$. All subgroup $(0,1)$-regular sets of $\dn$ were determined in \cite[Theorem 2.11]{Huang18}. Part (b) of the following theorem generalizes this result to subgroup $(0,k)$-regular sets of $\dn$ for odd integers $k$.

\begin{thm}
\label{th:th6.2}
Let $\dn$ be the dihedral group of order $2n \ge 6$ as given in \eqref{eq:D2n}, and let $H$ be a proper subgroup of $\dn$. Then the following statements hold for any integer $k$ between $1$ and $|H|$:
\begin{enumerate}[\rm (a)]
\item\label{it:it6.21} if $k$ is even, then $H$ is a $(0, k)$-regular set of $\dn$;
 \item\label{it:it6.22} if $k$ is odd, then $H$ is a $(0, k)$-regular set of $\dn$ if and only if $H\nleq \langle a \rangle$, or $H \leq \langle a \rangle$ with at least one of $|H|$ and $\frac{n}{|H|}$ odd.
 \end{enumerate}
\end{thm}

\begin{proof}
We deal with the two types of proper subgroups of $\dn$ separately.

\smallskip
\textsf{Claim 1.} If $H=\langle a^t\rangle$ with $t$ dividing $n$, then the following hold for any integer $k$ between $1$ and $|H|$: (i) if $k$ is even, then $H$ is a $(0, k)$-regular set of $\dn$; (ii) if $k$ is odd, then $H$ is a $(0,k)$-regular set of $\dn$ if and only if $t$ or $\frac{n}{t}$ is odd.
\smallskip

In fact, since $H =\langle a^t\rangle$ is a normal subgroup of $\dn$, statement (i) follows from Theorem~\ref{th:th5.3}. Note that $|H| = \frac{n}{t}$. By Corollary~\ref{co:co5.0}, if $k$ is odd, then $H$ is a $(0, k)$-regular set of $\dn$ if and only if $H$ is a $(0,1)$-regular set of $\dn$. By \cite[Lemma 2.9]{Huang18}, $H$ is a $(0,1)$-regular set of $\dn$ if and only if $t$ or $\frac{n}{t}$ is odd. Combining these, we obtain (ii) and hence establish Claim 1.

\smallskip
\textsf{Claim 2.} If $H=\langle a^t, a^sb\rangle$ with $t$ dividing $n$ and $0\leq s\leq n-1$, then $H$ is a $(0, k)$-regular set of $\dn$ for any integer $k$ between $1$ and $|H|$.
\smallskip

By Lemma~\ref{co:co2.1}, to prove this claim it suffices to construct $k$ pairwise disjoint subsets $S_1, S_2, \ldots, S_k$ of $\dn \setminus \{e\}$ such that their union is inverse-closed and $S_i \cup \{e\}$ is a right transversal of $H$ in $\dn$ for $1\leq i\leq k$. In fact, since $t$ divides $n$, we have $n = tm$ for some integer $m$. Since $|\langle a\rangle| = n$, we have $|\langle a^t\rangle| = m$ and hence $|H| = 2m$. Define
$$
R_i=\{a^{it+s-j}b: j=1,2,\ldots,t-1\},\;\, 0\leq i\leq m-1
$$
$$
T_i=\{a^{it+j}:j=1,2,\ldots, t-1\},\;\, 0\leq i\leq m-1.
$$
Note that these are $2m$ pairwise disjoint subsets of $\dn \setminus \{e\}$. A simple computation shows that, for each $1\leq i\leq m-1$, both $R_i\cup\{e\}$ and $T_i\cup\{e\}$ are right transversals of $H$ in $\dn$, and moreover $R^{-1}_i=R_i$ and $(T_i\cup T_{m-1-i})^{-1}=T_i\cup T_{m-1-i}$. If $1 \leq k \leq m$, we set
\[
S=R_0\cup R_1\cup\cdots\cup R_{k-1};
\]
and if $m < k \le |H|$, we set
\[
S=T_0\cup T_1\cup\cdots \cup T_{m-2}\cup T_{m-1}\cup R_{k-m-1}\cup R_{k-m-2}\cup\cdots\cup R_{1}\cup R_{0}.
\]
In either case, by Lemma~\ref{co:co2.1}, $H$ is a $(0, k)$-regular set in $\cay(\dn, S)$ and therefore a $(0, k)$-regular set of $\dn$. This establishes Claim 2.

Combining Claims 1 and 2, we obtain part (\ref{it:it6.21}) and the sufficiency in part (\ref{it:it6.22}).

Finally, for any proper subgroup $H$ of $\dn$, we have either $H\nleq \langle a \rangle$ or $H = \langle a^t\rangle \leq \langle a \rangle$ for some $t$ dividing $n$. Moreover, if $H = \langle a^t\rangle$ is a $(0, k)$-regular set of $\dn$ for some odd integer $k$ between $1$ and $|H| = \frac{n}{t}$, then by (ii) in Claim 1, either $t$ or $\frac{n}{t}$ is odd. This establishes the necessity in part (\ref{it:it6.22}) and hence completes the proof.
\end{proof}

The generalized quaternion group $\bn$ of order $4n$ is defined as
\begin{equation}
\label{eq:bn}
\bn=\langle a,b\ |\ a^n=b^2, a^{2n}=e, b^{-1}ab=a^{-1}\rangle.
\end{equation}
It is known that the subgroups of $\bn$ are $\langle a^t\rangle$ with $t$ dividing $2n$ and $\langle a^t, a^sb\rangle$ with $t$ dividing $2n$ and $0\leq s\leq t-1$. All subgroup $(0,1)$-regular sets of $\bn$ were determined by Ma \textit{et al.} in \cite[Theorem 1.7]{Ma19}. Part (b) of the following theorem extends this result to subgroup $(0,k)$-regular sets of $\bn$ for odd integers $k$.

\begin{thm}
\label{th:th6.4}
Let $\bn$ be the generalized quaternion group of order $4n \ge 8$ as given in \eqref{eq:bn}, and let $H$ be a proper subgroup of $\bn$. Then the following statements hold for any integer $k$ between $1$ and $|H|$:
\begin{enumerate}[\rm (a)]
\item if $k$ is even, then $H$ is a $(0, k)$-regular set of $\bn$;
\item if $k$ is odd, then $H$ is a $(0, k)$-regular set of $\bn$ if and only if either $H=\langle a^t\rangle$ with $t$ a divisor of $2n$ such that $\frac{2n}{t}$ is odd, or $H=\langle a^t, a^s b\rangle$ with $t$ an odd divisor of $2n$ and $0\leq s\leq t-1$.
 \end{enumerate}
\end{thm}

\begin{proof}
We handle the two types of proper subgroups of $\bn$ separately.

\smallskip
\textsf{Claim 1.} If $H=\langle a^t\rangle$ with $t$ dividing $2n$, then the following hold for any integer $k$ between $1$ and $\frac{2n}{t}$: (i) if $k$ is even, then $H$ is a $(0, k)$-regular set of $\bn$; (ii) if $k$ is odd, then $H$ is a $(0,k)$-regular set of $\bn$ if and only if $\frac{2n}{t}$ is odd.
\smallskip

In fact, since $H = \langle a^t\rangle $ is a normal subgroup of $\bn$,  by Theorem~\ref{th:th5.3}, $H$ is a $(0, k)$-regular set of $\bn$ for any even integer $k$ between $1$ and $\frac{2n}{t}$. By Corollary~\ref{co:co5.0}, for any odd integer $k$ between $1$ and $\frac{2n}{t}$, $H$ is a $(0, k)$-regular set of $\bn$ if and only if it is a $(0,1)$-regular set of $\bn$, which, by \cite[Theorem 1.7]{Ma19}, occurs if and only if $\frac{2n}{t}$ is odd. This proves Claim 1.

\smallskip
\textsf{Claim 2.} If $H=\langle a^t, a^sb\rangle$ with $t$ dividing $2n$ and $0\leq s\leq t-1$, then the following hold for any integer $k$ between $1$ and $|H|$: (i) if $k$ is even, then $H$ is a $(0,k)$-regular set of $\bn$; (ii) if $k$ is odd, then $H$ is a $(0,k)$-regular set of $\bn$ if and only if $t$ is odd.
\smallskip

Since any subgroup of $\bn$ not contained in $\langle a\rangle$ contains $a^n=b^2$, in the proof of Claim 2, we may assume without loss of generality that $t$ divides $n$, say, $n=mt$ for some integer $m$. Then $|H| = \frac{4n}{t} = 4m$ and $|G:H| = t$. Define
\begin{equation}
\label{eq:Ri}
R_i = \{a^{it+j}: j = 1, 2, \ldots, t-1\}, \;\, 0 \leq i \leq 2m-1
\end{equation}
\begin{equation}
\label{eq:Ti}
T_i = \{a^{it+s+j}b: j = 1, 2, \ldots, t-1\},\;\, 0\leq i\leq m-2
\end{equation}
and
$$
T=\{ab,a^2b,\ldots,a^{s-2}b,a^{s-1}b, a^{-t+s+1}b^3,a^{-t+s+2}b^3,\ldots,a^{-1}b^{3},b^3\}.
$$
Then
$$
T_i^{-1}=\{a^{it+s+j}b^3: j = 1, 2, \ldots, t-1\}, \;\, 0\leq i\leq m-2
$$
$$
T^{-1}=\{ab^3,a^2b^3,\ldots,a^{s-2}b^3,a^{s-1}b^3, a^{-t+s+1}b,a^{-t+s+2}b,\ldots,a^{-1}b,b\}.
$$
Note that $R_i^{-1} = R_{2m-1-i}$ for $0\leq i\leq 2m-1$. A simple computation shows that $R_i\cup\{e\}$ is a right transversal of $H$ in $\bn$ for $0\leq i\leq 2m-1$. Similarly, $T\cup\{e\}, T^{-1}\cup\{e\}, T_i\cup\{e\}$ and $T_{i}^{-1}\cup\{e\}$ for $0\leq i\leq m-2$ are all right transversals of $H$ in $\bn$.

For any even integer $k$ between $2$ and $2m$, set
$$
S = R_{0}\cup R_1\cup\cdots\cup R_{\frac{k}{2}-2}\cup R_{\frac{k}{2}-1}\cup R_{2m-\frac{k}{2}}\cup R_{2m-\frac{k}{2}+1}\cup\cdots\cup R_{2m-2}\cup R_{2m-1};
$$
and for any even integer $k$ between $2m+2$ and $4m$, set
$$
S=T\cup T^{-1}\cup T_0\cup\cdots\cup T_{m-2}\cup T_0^{-1}\cup\cdots\cup T_{m-2}^{-1}\cup R_{\frac{k}{2}-m-1}\cup\cdots\cup R_{0}\cup R_{3m-\frac{k}{2}}\cup\cdots\cup R_{2m-1}.
$$
In either case $S$ is inverse-closed. Thus, by Lemma~\ref{co:co2.1}, for any even integer $k$ between $2$ and $|H|$, $H$ is a $(0, k)$-regular set in $\cay(\bn, S)$ and hence a $(0, k)$-regular set of $\bn$. This proves (i) in Claim 2.

Now we prove (ii) in Claim 2. Suppose that $H$ is a $(0,k)$-regular set of $\bn$ for some odd $k$. Then by Lemma~\ref{co:co2.1} there exists an inverse-closed subset $S$ of $\bn \setminus \{e\}$ with $|S| = k(t-1)$ such that $\cay(\bn, S)$ admits $H$ as a $(0,k)$-regular set. If $t$ is even, then $|S|$ is odd. Since $S^{-1} = S$ and $a^n$ is the only involution in $\bn$, we must have $a^n \in S$, but this is a contradiction. Hence $t$ must be odd and the ``only if'' part in (ii) is established. To prove the ``if'' part in (ii), assume that $t$ is odd, say, $t = 2r + 1$ for some integer $r$. Define
$$
R=\{a, a^{-1}, a^2, a^{-2}, \ldots, a^{r}, a^{-r}\}
$$
$$R'=\{a^{r+1}, a^{-r-1}, a^{r+2}, a^{-r-2}, \ldots, a^{t-1}, a^{1-t}\}.$$
Then both $R$ and $R'$ are inverse-closed, and both $R \cup \{e\}$ and $R' \cup \{e\}$ are right transversals of $H$ in $\bn$. For any odd integer $2l+1$ between $1$ and $2m-1$, define
\[
S=R_1 \cup \cdots \cup R_l \cup R_{2m-1-l} \cup \cdots \cup R_{2m-2} \cup R;
\]
and for any odd integer $2l+1$ between $2m+1$ and $4m-1$, define
\[
S=T\cup T^{-1}\cup T_0\cup\cdots\cup T_{m-2}\cup T_0^{-1}\cup\cdots\cup T_{m-2}^{-1}\cup R_1\cup\cdots\cup R_{l-m}\cup R_{3m-1-l}\cup\cdots\cup R_{2m-2}\cup R.
\]
In either case $S$ is inverse-closed. Thus, by Lemma~\ref{co:co2.1}, for any odd integer $k$ between $1$ and $|H|$, $H$ is a $(0, k)$-regular set in $\cay(\bn,S)$ and hence a $(0,k)$-regular set of $\bn$. This completes the proof of Claim 2.

The desired result follows from Claims 1 and 2 immediately.
\end{proof}

\section{$(0,k)$-regular sets in cubelike graphs}
\label{sec:seccube}

In this and the next sections we use $V(n,p)$ to denote the $n$-dimensional linear space over $\FFF_p$, where $n \ge 1$, $p$ is a prime, and $\FFF_p$ is the finite field with $p$ elements. We identify the elementary abelian $p$-group $\zp$ with the additive group of $V(n,p)$. The vectors of $V(n,p)$ are treated as column vectors and the zero vector of $V(n,p)$ is denoted by $\mathbf{0}_n$. Given an $m\times m$ matrix $M$ over $\FFF_p$, we use $M_i$ to denote the $i$-th column of $M$, for $1\leq i\leq m$.
For any integer $1\le t\leq n$, let $P_{t\times n}(p)= (\mathbf{I}\;\,\mathbf{0})$ be the $t \times n$ matrix over $\FFF_p$ whose first $t$ columns form the $t \times t$ identity matrix $\mathbf{I}$ and last $n-t$ columns form the $t \times (n-t)$ all-$0$ matrix $\mathbf{0}$.

The $n$-dimensional cube $Q_n$ is the Cayley graph of $\zzz$ with respect to the set of vectors with exactly one nonzero coordinate. In general, any Cayley graph of $\zzz$ is called \cite{god01} a {\em cubelike graph}. A $(0,1)$-regular set in $Q_n$ is precisely a perfect $1$-error correcting binary code of length $n$ in coding theory. It is known (see, for example, \cite{heden08}) that $(0,1)$-regular sets in $Q_n$ exist if and only if $n = 2^t-1$ for some $t \ge 1$, and most likely Hamming \cite{hamming50} was the first person who constructed $(0,1)$-regular sets in $Q_{2^t-1}$. In this section, we generalize these results to $(0,k)$-regular sets in hypercubes $Q_n$. We first give a sufficient condition for a connected cubelike graph admitting $(0,k)$-regular sets. For any connected cubelike graph $\cay(\zzz,S)$, denote $U=\pmat{\bu_1 & \bu_2 & \cdots & \bu_d}$ the $n\times d$ matrix with rank $n$, where $S=\{\bu_1, \bu_2, \ldots, \bu_d\}$ for $1\leq d$. Let $N$ be the $t \times (2^t-1)$ matrix over $\FFF_2$ whose columns are the nonzero vectors of $V(t,2)$. For any integer $1\leq k$, denote $N(k)=\left(\underbrace{
\begin{array}{cccc}
N  & N & \cdots & N
\end{array}}_{k \text{ times}}
\right)$.

\begin{lem}
\label{th:th5.1}
Let $n$ and $k$ be integers with $n > k \ge 1$. Let $\cay(\zzz,S)$ be a connected cubelike graph with $|S|=d$. If  there is  an integer $1\le t\le n$ such that $d=(2^t-1)k$ and there is a non-singular matrix $R$ such that $P_{t\times n}(2)RU=N(k)$, then $\cay(\zzz,S)$ admits a $(0,k)$-regular set.
\end{lem}

\begin{proof}
Let $\Gamma=\cay(\zzz,S)$ be a connected cubelike graph.  Assume that $|S| = d := (2^t-1)k \le 2^n - 1$ for some $1 \le t \le n$. Write $S = \{\bu_1, \bu_2, \ldots, \bu_d\} \subseteq V(n,2) \setminus \{\mathbf{0}_n\}$.  Set $P=P_{t\times n}(2)$.
Since there is an integer  $1\le t\le n$ such that $d=(2^t-1)k$ and there is a non-singular matrix $R$ such that $PRU=N(k)$, then  there exists an $n\times d$ matrix $Q$ over $\FFF_2$ with rank $n$ such that
\[
PQ=\left(\underbrace{
\begin{array}{cccc}
N  & N & \cdots & N
\end{array}}_{k \text{ times}}
\right).
\]
In fact, we can take $Q=RU$, where
$
U = \pmat{\bu_1 & \bu_2 & \cdots & \bu_d}
$
 is the $n \times d$ matrix with rank $n$.  Set $M=PR$. Since $P$ has rank $t$, $R$ is a non-singular matrix, $M$ is a $t\times n$ matrix over $\FFF_2$ with rank $t$. Hence the null space of $M$, namely
\begin{equation}
\label{eq:Wn2}
W=\{\bw\in V(n,2)\ : \ M \bw=\mathbf{0}_n\},
\end{equation}
is an $(n-t)$-dimensional subspace of $V(n,2)$. In the rest proof we will show that $W$ is a $(0,k)$-regular set in $\Gamma$.

\smallskip
\textsf{Claim 1.} $W$ is an independent set of $\Gamma$.
\smallskip

Suppose to the contrary that there exist distinct vertices $\bv, \bw \in W$ which are adjacent in $\Ga$. Then $\bw = \bv + \bu_i$ for some $\bu_i \in S$. Since $M=PR$, $Q=RU$ and $\bu_i$ is a column of $U$, we have
\[
\mathbf{0}_n = M\bw = M(\bv + \bu_i) = M\bv + M\bu_i = M\bu_i = (PR)\bu_i =P(R\bu_i)=P(Q_i)=(PQ)_i,
\]
which contradicts the fact that all column vectors of $PQ$ are nonzero. This proves Claim 1.

\smallskip
\textsf{Claim 2.} Every vertex $\bv \in V(n,2) \setminus W$ has at most $k$ neighbours in $W$.
\smallskip

Suppose to contrary that some $\bv \in V(n,2) \setminus W$ has (at least) $k+1$ distinct neighbours $\bw_1,\bw_2,\ldots,\bw_k,\bw_{k+1}$ in $W$. Then
 \[
 \bv=\bw_1+\bu_{i_1}=\bw_2+\bu_{i_2}=\cdots=\bw_k+\bu_{i_k}=\bw_{k+1}+\bu_{i_{k+1}}
 \]
for $k+1$ distinct elements $\bu_{i_1},\bu_{i_2},\ldots, \bu_{i_{k}}, \bu_{i_{k+1}}$ of $S$. Hence
$$
\mathbf{0}_n = M(\bw_j - \bw_1) = M(\bu_{i_1} - \bu_{i_j}) = M\bu_{i_1} - M\bu_{i_j}
$$
for $2\leq j\leq k+1$. Therefore, $M\bu_{i_1}=M\bu_{i_2}=\cdots =M\bu_{i_{k+1}}$. Since $Q=RU$, $M=PR$ and $\bu_{i_j}$ is a column of $U$, we have $M\bu_{i_j}=P(R\bu_{i_j})=P(Q_{i_j})=(PQ)_{i_j}$ for $1\leq j\leq k+1$. So there exist $k$ different columns of $PQ$ which are equal to $(PQ)_{i_1}$, but this contradicts the fact that $PQ$ does not contain $k+1$ identical columns. This proves Claim 2.

\smallskip
\textsf{Claim 3.} For any $\bv \in V(n,2) \setminus W$, there exists an element $\bu_i \in S$ such that $\bv \in W+\bu_i$.
\smallskip

In fact, since $\Ga$ is connected, any $\bv \in V(n,2) \setminus W$ can be expressed as $\bv = \bu_{i_1}+\bu_{i_2}+\cdots+\bu_{i_r}$ for some (not necessarily distinct) elements $\bu_{i_1}, \bu_{i_2}, \ldots, \bu_{i_r} \in S$.
Thus,
$$
M\bv = M\bu_{i_1}+M\bu_{i_2}+\cdots+M\bu_{i_r} = (PQ)_{i_1}+(PQ)_{i_2}+\cdots+(PQ)_{i_r}.
$$
Since $M\bv \ne \mathbf{0}_n$ and all nonzero vectors of $V(t,2)$ appear in the columns of $N$, $M\bv$ occurs as a column of $N$, say, the $i$-th column of $N$ for some $i$ between $1$ and $2^t-1$. Then $M\bv$ also occurs as the $i$-th column of $PQ$, that is, $M\bv=(PQ)_i$.  Since  $Q=RU$, $M=PR$, taking the $i$-th column $\bu_i$ of $U$, we have $M\bv=(PQ)_{i}=M\bu_i$. Hence $\bu_i-\bv\in W$ and $\bv \in W+\bu_i$ as desired.

\smallskip
\textsf{Claim 4.} For any $\bv \in V(n,2) \setminus W$, there exist exactly $k$ distinct elements $\bu_{i_1}, \bu_{i_2}, \ldots, \bu_{i_k}$ of $S$ such that $\bv \in W+\bu_{i_j}$ for $1\leq j \leq k$.
\smallskip

In fact, by Claim 3, for any $\bv \in V(n,2) \setminus W$, there exists an element $\bu_i \in S$ such that $\bv \in W+\bu_i$. Since $(PQ)_h = (PQ)_l$ for any $1\leq h \leq d$ and $l\equiv h + (2^t-1)j \pmod{d}$ with $0\leq j\leq k-1$, we have $\bv \in W+\bu_{l}$ for all $l \in \{1, 2, \ldots, d\}$ such that $l \equiv i+(2^t-1)j\; (\mathrm{mod}\; d)$ for some $0\leq j\leq k-1$. This means that there are at least $k$ distinct elements $\bu_{l} \in S$ such that $\bv \in W + \bu_{l}$. Combining this with Claim 2, we obtain Claim 4.

By Claims 1 and 4, $W$ is a $(0, k)$-regular set in $\Gamma$. This completes the proof.
\end{proof}

Next, we give the main result in this section.


\begin{thm}
\label{cor:Qn}
Let $n$ and $k$ be integers with $n > k \ge 1$. Then the following statements are equivalent:
\begin{enumerate}[\rm (a)]
\item\label{it:it31}
$Q_n$ admits a $(0,k)$-regular set;
\item\label{it:it32}
$n=(2^{t}-1)k$ for some integer $t \ge 1$;
\item\label{it:it33}
$k$ divides $n$ and $Q_n$ is a $k$-cover of $K_{\frac{n}{k}+1}$.
\end{enumerate}
\end{thm}

\begin{proof}
 Let $Q_n=\cay(\zzz,S)$, where $S=\{\be_i: 1\le i\le n\}$ with $\be_1, \be_2, \ldots, \be_n$ the standard basis of $\mathbb{Z}_p^n$.

 \eqref{it:it31} $\Rightarrow$ \eqref{it:it32} If $Q_n$ admits a $(0,k)$-regular set, then by (\ref{le:le1.1}), $1+\frac{|S|}{k}$ divides $2^n$ and hence $\Ga$ has degree $|S| = (2^t-1)k \le 2^n - 1$ for some $1 \le t \le n$.

\eqref{it:it32} $\Rightarrow$ \eqref{it:it31} By Lemma \ref{th:th5.1}, it is sufficient to prove that there is an $n\times n$ non-singular matrix $R$ such that $P_{t\times n}(2)RU=N(k)$.
Since $S=\{\be_1, \be_2, \ldots, \be_n\}$, thus $U$ is the $n\times n$ identity matrix with rank $n$.
Set $P=P_{t\times n}(2)$. Note that $P$ is a $t\times n$ matrix with rank $t$ and $N$ is a $t\times 2^t-1$ matrix.
Since $1 \leq t \leq n$ and $n = (2^t-1)k$, there exists an $n\times n$ matrix $Q$ over $\FFF_2$ with rank $n$ such that
\[
PQ=\left(\underbrace{
\begin{array}{cccc}
N  & N & \cdots & N
\end{array}}_{k \text{ times}}
\right).
\]
Since $Q$ and $U$ have the same rank and dimension, there are two $n\times n$ non-singular matrices $R$, $L$ over $\FFF_2$ such that $Q=RUL$. Since $U$ is the identity matrix, thus $Q=RLU=R'U$ with $R'=RL$. Combining this with Lemma \ref{th:th5.1}, we obtain the result.

\eqref{it:it31} $\Leftrightarrow$ \eqref{it:it33} It follows from Theorem \ref{th:th2.0}.
\end{proof}

Since each hypercube graph with  $(0,k)$-regular sets admits a non-singular matrix $R$ such that $P_{t\times n}(2)RU=N(k)$, thus  the proof of Lemma \ref{th:th5.1} also gives a construction of a subgroup $(0,k)$-regular set in any hypercube with order $2^n$ and degree $(2^t-1)k \le 2^n - 1$. In fact, this subgroup $(0,k)$-regular set (as defined in \eqref{eq:Wn2}) is a ``linear'' $(0,k)$-regular set in the sense that it is a subspace of the vector space $V(n,2)$. We illustrate this construction by the following example.

\begin{exa}
Since $6 = (2^2 - 1) \cdot 2$, by Corollary \ref{cor:Qn}, $Q_6$ admits $(0,2)$-regular sets. Following the proof of Theorem \ref{th:th5.1}, we have $t=2$ and
$$
U=\pmat{1 & 0 & 0 & 0 & 0 & 0 \\
0 & 1 & 0 & 0 & 0 & 0\\
0 & 0 & 1 & 0 & 0 & 0\\
0 & 0 & 0 & 1 & 0 & 0\\
0 & 0 & 0 & 0 & 1 & 0\\
0 & 0 & 0 & 0  & 0  & 1},\,\,
Q=R=\pmat{0 & 1 & 1 & 0 & 1 & 1 \\
1 & 0 & 1 & 1 & 0 & 1\\
1 & 0 & 0 & 0 & 0 & 0\\
0 & 1 & 0 & 1 & 0 & 0\\
0 & 0 & 1 & 0 & 1 & 0\\
0 & 0 & 0 & 0  & 0  & 1},
$$
$$
P=\pmat{1 & 0 & 0 & 0 & 0 & 0 \\
0 &1 & 0 & 0 & 0 & 0}, \,\,
M = \pmat{0 & 1 & 1 & 0 & 1 & 1 \\
1 & 0 & 1 & 1 & 0 & 1}.
$$
The null space $W = \{\bw\in V(6,2)\ : \ M\bw=\mathbf{0}_6\}$ of $M$ is a $(0,2)$-regular set in $Q_6$. Solving
$$
\left\{
\begin{array}{ll}
w_2+w_3+w_5+w_6=0\\
w_1+w_3+w_4+w_6=0,
\end{array}
\right.
$$
we obtain that $W$ is the set of column vectors of the matrix:
$$
\left(
 \begin{array}{cccccccccccccccc}
   0 & 0 & 0 & 0 & 0 & 0 & 0 & 0 & 1 & 1 & 1 & 1 & 1 & 1 & 1 & 1\\
   0 & 0 & 0 & 0 & 1 & 1 & 1 & 1 & 0 & 0 & 0 & 0 & 1 & 1 & 1 & 1\\
   0 & 0 & 1 & 1 & 1 & 1 & 0 & 0 & 0 & 0 & 1 & 1 & 1 & 1 & 0 & 0\\
   0 & 1 & 0 & 1 & 0 & 1 & 0 & 1 & 0 & 1 & 0 & 1 & 0 & 1 & 0 & 1\\
   0 & 1 & 0 & 1 & 1 & 0 & 1 & 0 & 1 & 0 & 1 & 0 & 0 & 1 & 0 & 1\\
   0 & 1 & 1 & 0 & 1 & 0 & 0 & 1 & 1 & 0 & 0 & 1 & 0 & 1 & 1 & 0
   \end{array}
   \right).
$$
\end{exa}

\section{$(0,k)$-regular sets in the Lee metric}
\label{sec:secp}

The Cartesian product $\Gamma_1 \Box \Gamma_2 \Box \cdots \Box \Gamma_n$ of $n$ graphs $\Gamma_1, \Gamma_2, \ldots, \Gamma_n$ is the graph with vertex set $V(\Gamma_1)\times V(\Gamma_2)\times\cdots\times V(\Gamma_n)$ such that two vertices $(u_1,u_2,\ldots,u_n), (v_1,v_2,\ldots,v_n)$ are adjacent if and only if there is exactly one subscript $i$ such that $u_i\neq v_i$, and for this $i$, $u_i$ and $v_i$ are adjacent in $\Gamma_i$. Denote by $C_{q}^{\Box n}$ the Cartesian product of $n$ copies of the cycle $C_q$ of length $q$, where $n \ge 1$ and $q \ge 3$ are integers. Obviously, $C_{q}^{\Box n}$ is a $2n$-regular graph. One can see that $C_{q}^{\Box n}$ is the Cayley graph $\cay(\mathbb{Z}_q^n, S)$, where $S=\{\be_i, -\be_i\ :\ 1\leq i\leq n\}$ with $\be_i$ is the element of $\mathbb{Z}_q^n$ with $i$-th coordinate $1$ and all other coordinates $0$. The distance in $C_{q}^{\Box n}$ between two vertices is exactly the Lee distance between the corresponding codewords. The Lee ball with radius $e$ and center $\mathbf{x} \in \ZZZ_q^n$ is the set of elements of $\ZZZ_q^n$ with Lee distance at most $e$ to $\mathbf{x}$. A code $C \subseteq \ZZZ_q^n$ is called a $q$-ary perfect $e$-code under the Lee metric if the Lee balls of radius $e$ with centers in $C$ form a partition of $\ZZZ_q^n$. The well-known Golomb-Welch conjecture asserts that there is no $q$-ary perfect $e$-codes of length $n$ under the Lee metric for $n > 2$, $e > 1$ and $q \ge 2e+1$. A central problem for Lee codes, this 50-year-old conjecture is still wide open \cite{horak18} in its general form.

Note that $q$-ary perfect $1$-codes of length $n$ under the Lee metric are exactly $(0, 1)$-regular sets in $C_{q}^{\Box n}$. In \cite{gol70}, Golomb and Welch constructed $p$-ary perfect $1$-codes of length $n = \frac{p-1}{2}$ under the Lee metric, where $p$ is an odd prime. More generally, a special case of a result \cite[Theorem 15]{ALB09} proved by AlBdaiwi, Horak and Milazzo asserts that for any odd prime $p$, a $p$-ary linear perfect $1$-code of length $n$ under the Lee metric exists if and only if $2n = p^t-1$ for some $t$. The following theorem generalizes this result to $(0,k)$-regular sets in $C_{p}^{\Box n}$.

\begin{thm}
\label{th:th5.2}
Let $p$ be an odd prime and $n, k$ positive integers. Then $C_{p}^{\Box n}$ admits $(0, k)$-regular sets if and only if $2n = (p^t-1)k \le p^n - 1$ for some $1 \le t \le n-1$.
\end{thm}

\begin{proof}
We identify $C_{p}^{\Box n}$ with $\cay(\zp,S)$, where $S=\{\be_i, -\be_i\ :\ 1\leq i\leq n\}$, with $\be_1, \be_2, \ldots, \be_n$ the standard basis of $\mathbb{Z}_p^n$. We also identify $\zp$ with the additive group of the $n$-dimensional linear space $V(n,p)$ over $\FFF_p$.

If $C_{p}^{\Box n}$ admits $(0,k)$-regular sets, then by (\ref{le:le1.1}), $1+\frac{2n}{k}$ divides $p^n$ and hence $C_{p}^{\Box n}$ has degree $2n = (p^t-1)k \le p^n - 1$ for some $1 \leq t \le n-1$, establishing the necessity.

Now we prove the sufficiency. Suppose that $2n = (p^t-1)k \le p^n - 1$ for some $1 \le t \le n-1$. Set $P = P_{t\times n}(p)= (\mathbf{I}\;\, \mathbf{0})$. Take $N$ to be a $t \times \frac{(p^t-1)}{2}$ matrix over $\FFF_p$ with columns nonzero vectors in $V(t, p)$ such that $\bu$ is a column of $N$ if and only if $-\bu$ is not a column of $N$. Since $1 \leq t \le n-1$, there exists an $n\times n$ matrix $Q$ over $\FFF_p$ with rank $n$ such that
\[
PQ=\left(\underbrace{
 \begin{array}{cccc}
   N & N & \cdots & N
   \end{array}}_{k \text{ times}}
   \right).
\]
Let $M = PQ$. Since $P$ is a $t \times n$ matrix with rank $t$ and $Q$ is an $n\times n$ matrix with rank $n$, $M$ is a $t \times n$ matrix over $\FFF_p$ with rank $t$. Hence the null space of $M$, namely
\begin{equation}
\label{eq:Wnp}
W = \{\bw\in V(n,p)\ :\ M\bw = \mathbf{0}_n\},
\end{equation}
is an $(n-t)$-dimensional subspace of $V(n,p)$. We now prove that $W$ is a $(0,k)$-regular set in $\Gamma$.

\smallskip
\textsf{Claim 1.} $W$ is an independent set of $C_{p}^{\Box n}$.
\smallskip

Suppose otherwise. Then there exist distinct $\bv, \bw \in W$ adjacent in $C_{p}^{\Box n}$. So $\bw = \bv \pm \be_i$ for some $\be_i \in S$. Thus,
$$
\mathbf{0}_n = M\bw = M(\bv \pm \be_i) = M\bv \pm M\be_i = (PQ)\be_i = (PQ)_i,
$$
which contradicts the fact that all column vectors of $PQ$ are nonzero. This proves Claim 1.

\smallskip
\textsf{Claim 2.} Every vertex $\bv \in V(n,p) \setminus W$ has at most $k$ neighbours in $W$.
\smallskip

Suppose to the contrary that some $\bv \in V(n,p) \setminus W$ has (at least) $k+1$ distinct neighbours $\bw_1, \bw_2, \ldots, \bw_k, \bw_{k+1}$ in $W$. Then there exist $k+1$ distinct vectors $\bu_1, \bu_2, \ldots, \bu_k, \bu_{k+1}$ in $S$ such that
\[
\bv = \bw_1+\bu_1 = \bw_2 + \bu_2 = \cdots = \bw_k + \bu_k = \bw_{k+1} + \bu_{k+1}.
\]
Hence
$$
\mathbf{0}_n = M(\bw_j - \bw_1) = M(\bu_{1} - \bu_{j}) = M\bu_{1} - M\bu_{j}
$$
for $2\leq j\leq k+1$. Since $\bu_{1} = \pm \be_{i}$ for some $i$, it follows that there exist $k$ different columns of $PQ$ which are equal to $\pm (PQ)_{i}$, which contradicts the fact that $PQ$ does not contain $k+1$ identical columns. This proves Claim 2.

\smallskip
\textsf{Claim 3.} For any $\bv \in V(n,p) \setminus W$, there exists an element $\be_i \in S$ such that $\bv \in W \pm \be_i$.
\smallskip

In fact, since $C_{p}^{\Box n}$ is connected, any $\bv \in V(n,p) \setminus W$ can be expressed as $\bv = \bu_{i_1}+\bu_{i_2}+\cdots+\bu_{i_r}$, where $\bu_{i_j} = \pm \be_{i_j} \in S$ for $1 \le j \le r$. (Note that $i_1, i_2, \ldots, i_r$ are not necessarily distinct.) Note that for any nonzero vector $\bu$ of $V(t,p)$, either $\bu$ or $-\bu$ is a column of $N$. Thus, if $\bu$ is a column of $N$, then $\bu$ is a column of $PQ$; otherwise, $-\bu$ is a column of $PQ$. Since $M\bv \ne \mathbf{0}_n$, it follows that $M\bv$ is a column of $PQ$ or $-PQ$, say, $M\bv=\pm (PQ)_i$ for some $i$. Since $\pm(PQ)_i = M\bu_i$ where $\bu_i=\pm \be_i$, it follows that $\bv - \bu_i \in W$. That is, $\bv \in W \pm \be_i$ as desired.

\smallskip
\textsf{Claim 4.} For any $\bv \in V(n,p) \setminus W$, there exist exactly $k$ distinct elements $\be_{i_1}, \be_{i_2}, \ldots, \be_{i_k}$ in $S$ such that $\bv \in W \pm \be_{i_j}$ for $1\leq j\leq k$.
\smallskip

In fact, by Claim 3, for any $\bv \in V(n,p) \setminus W$, there exists an element $\be_i \in S$ such that $\bv \in W \pm \be_i$. Since $(PQ)_h = (PQ)_l$ for any $1\leq h \leq n$ and $l \equiv h+\frac{(p^t-1)}{2}j \pmod{n}$ with $0\leq j\leq k-1$, we have $\bv \in W \pm \be_{l}$ for all $l \in \{1, 2, \ldots, n\}$ such that $l \equiv i + \frac{(p^t-1)}{2}j\; (\mathrm{mod}\; n)$ for some $0\leq j\leq k-1$. Therefore, there are at least $k$ distinct elements $\be_l$ of $S$ such that $\bv \in W \pm \be_{l}$. Combining this with Claim 2, we obtain Claim 4.

By Claims 1 and 4, $W$ is a $(0, k)$-regular set in $C_{p}^{\Box n}$. This completes the proof.
\end{proof}

The proof above gives, for any odd prime $p$ and positive integers $n, k$ satisfying $2n = (p^t-1)k \le p^n - 1$ for some $1 \le t \le n-1$, a construction of a $(0, k)$-regular set (as defined in \eqref{eq:Wnp}) in $C_{p}^{\Box n}$, and moreover this $(0, k)$-regular set is ``linear'' as it is a subspace of the vector space $V(n,p)$.

\section{Concluding remarks}

 Note that Theorems \ref{th:th5.3} and \ref{th:th5.4} show two necessary and sufficient conditions for a normal subgroup of a group to be a $(0,k)$-regular set of the group. As seen in the proof of  Theorems \ref{th:th5.3} and \ref{th:th5.4}, the normal subgroup of a group plays a key role in obtaining the Cayley graph of the group such that the normal subgroup is a $(0,k)$-regular set of the Cayley graph for some integer $1\leq k$. If we consider that the subgroup is not normal, what happens? We need to keep thinking about it. Thus it would be interesting to study whether Theorems \ref{th:th5.3} and \ref{th:th5.4}  are still true if the subgroup is not normal.

As seen in Section \ref{sec:eq-partn}, a Cayley graph $\Gamma=\cay(G, S)$ admits a subgroup $(0, k)$-regular set $H$, then $\Gamma$ admits a $(-k)$-equitable partition with exactly $|G:H|$ blocks and with $H$ as one of the blocks. In particular, $\pi=\{H, V(\Gamma)\setminus H\}$ is a $(-k)$-equitable partition of $\Gamma$. Such equitable partition of $\Gamma$  is called a {\em perfect $2$-coloring} of $\Gamma$ \cite{fon07}. Perfect $2$-colorings of a graph have been studied extensively over many years, such as, the existence of perfect $2$-colorings of Johnson graphs $J(v,3)$ \cite{gav13}, Hamming graph \cite{bes21} and generalized Petersen graphs \cite{mehdi16,mehdi22}. Similar to the study of perfect $2$-colorings of some graphs, it would be interesting to give all possible regular sets of some graphs.

In addition,  completely regular subsets of a graph  are closely related to the perfect codes of the graph. Given a  graph $\Gamma$, a subset $V$ of $V(\Gamma)$ is called a {\em completely regular set} \cite{god93} of $\Gamma$ if $\pi(V)=\{V_0, V_1, \ldots, V_d\}$ is an equitable partition of $\Gamma$, where $S_i$ is the set of all vertices at distance $i$ from $S$ for  any integer $0\le i\le d$, and $d$ is the covering radius of $S$. Some known results using completely regular sets and equitable partitions can help us to  classify certain perfect codes of some graphs. What happens to these notations when we study regular sets of groups? This will be considered in our future work.

In Section \ref{sec:seccube}, we gave a sufficient condition for cubelike graphs admitting $(0,k)$-regular sets. Using the result, we obtained the necessary and sufficient condition for the existence of $(0,k)$-regular sets of hypercubes. It is also interesting to study the existence of $(0,k)$-regular sets of cubelike graphs.

\subsubsection*{Acknowledgement}

The authors would like to thank the anonymous reviews for their valuable comments which improved the paper.
Xu was supported by the National Natural Science Foundation of China (Grant No.\@ 12071194, 11901263).

\subsubsection*{Data deposition Information}

No data were used support this study.

\subsubsection*{Conflict of Interest}
The authors declare that they have no conflicts of interest.


\begin{thebibliography}{99}
\parskip -0.1cm
\bibitem{ALB09}
B.F. AlBdaiwi, P. Horak, L. Milazzo, Enumerating and decoding perfect linear Lee codes, Des. Codes Cryptogr. 52(2) (2009) 155--162.

\bibitem{Bai19}
R.A. Bailey, P.J. Cameron, A.L. Gavrilyuk, S. V. Goryainov, Equitable partitions of Latin-square graphs, J. Combin. Des. 27(3) (2019) 142--160.

\bibitem{Bai20}
R.A. Bailey, P.J. Cameron, S. Zhou, Equitable partitions of Cayley graphs, in preparation.



\bibitem{beh21}
A. Behajaina, R. Maleki, A. S. Razafimahatratra, On non-normal subgroup perfect codes,  Australas. J. Combin. 81(3) (2021) 474--479.

\bibitem{bes21} E.A. Bespalov, D.S. Krotov, A.A. Matiushev, A.A. Taranenko, K.V. Vorob'ev, Perfect $2$-colorings of Hamming graphs, J. Combin. Des. 29(6) (2021) 367--396.

\bibitem{Big01}
N. Biggs, Algebraic Graph Theory, Cambridge University Press, Cambridge, 2001.


\bibitem{cardoso19}
D.M. Cardoso, An overview of $(\kappa,\tau)$-regular sets and their applications,  Discrete Appl. Math. 269 (2019) 2--10.

\bibitem{Chen19}
J. Chen, Y. Wang, B. Xia, Characterization of subgroup perfect codes in Cayley graphs,  Discrete Math. 343(5) (2020) 111813.



\bibitem{de03}
I.J. Dejter, O. Serra, Efficient dominating sets in Cayley graphs.  Discrete Appl. Math. 129(2-3) (2003) 319--328.

\bibitem{deng14}
Y.-P. Deng, Efficient dominating sets in circulant graphs with domination number prime,  Inform. Process. Lett.  114(12) (2014) 700--702.

\bibitem{deng17}
Y.-P. Deng, Y.-Q. Sun, Q. Liu, H.-Q. Wang, Efficient dominating sets in circulant graphs,  Discrete Math. 340(7) (2017) 1503--1507.

\bibitem{feng17}
R. Feng, H. Huang, S. Zhou,  Perfect codes in circulant graphs, Discrete Math. 340(7) (2017) 1522--1527.

\bibitem{fon07}
D.G. Fon-Der-Flaas, Perfect $2$-colorings of a hypercube, Siberian Math. J. 48(4) (2007) 740--745. (Translated feom Sibirsk. Mat. Zh. 48(4) (2007) 923--930.)

\bibitem{gav13}
A.L. Gavrilyuk, S.V. Goryainov, On perfect $2$-colorings of Johnson graphs $J(v,3)$, J. Combin. Des. 21(6) (2013) 232--252.

\bibitem{god93} C. Godsil, Algebraic Combinatorics, Chapman \& Hall, New York, 1993.

\bibitem{god01}
C. Godsil, G. Royle, Algebraic Graph Theory, Graduate Texts in Mathematics, Springer, New York, 2001.

\bibitem{gol70}
S.W. Golomb, L. R. Welch, Perfect codes in the Lee metric and the packing of polyominoes,  SIAM J. Appl. Math. 18(2) (1970) 302--317.

\bibitem{hamming50}
R.W. Hamming, Error detecting and error correcting codes,  Bell. System Tech. J. 29(2) (1950) 147--160.

\bibitem{hhs981}
T.W. Haynes, S.T. Hedetniemi, P. J. Slater, Fundamentals of Domination in Graphs, Marcel Dekker, New York, 1998.

\bibitem{heden08}
O. Heden, A survey of perfect codes,  Adv. Math. Commum. 2(2) (2008) 223--247.

\bibitem{horak18}
P. Horak, D. Kim, 50 Years of the Golomb-Welch conjecture, IEEE Trans. Inform. Theory 64(4) (2018) 3048--3061.

\bibitem{Huang18}
H. Huang, B. Xia, S. Zhou, Perfect codes in Cayley graphs, SIAM J. Discrete Math  32(1) (2018) 548--559.


\bibitem{kwak98}
J.H. Kwak, J.-H. Chun, J. Lee, Enumeration of regular graph coverings having finite abelian covering transformation groups, SIAM J. Discrete Math. 11(2) (1998) 273--285.

\bibitem{Kwo14}
Y. S. Kwon, J. Lee, Perfect domination sets in Cayley graphs, Discrete  Appl. Math. 162 (2014) 259--263.

\bibitem{Lee01}
J. Lee, Independent perfect domination sets in Cayley graphs, J. Graph Theory 37(4) (2001) 213--219.

\bibitem{Ma19}
X. Ma, G.L. Walls, K. Wang, S. Zhou, Subgroup perfect codes in Cayley graphs,  SIAM J. Discrete Math. 32 (3) (2020) 1009--1931.

\bibitem{mar07}
C. Mart\'{i}nez, R. Beivide, E. Gabidulin, Perfect codes for metrics induced by circulant graphs,  IEEE Trans. Inform. Theory 53(9) (2007) 253--258.

\bibitem{mehdi16} A. Mehdi, K. Hamed, Perfect $2$-colorings of the generalized Petersen graph, Proc. Indian Acad. Sci. (Math. Sci.) 126(3) (2016) 289--294.


\bibitem{mehdi22} A. Mehdi, A. Efat, perfect $2$-colorings of the Johnson graph $J(9,4)$, Math. Sci. 16(2) (2022) 122-136.

\bibitem{ob07}
N. Obradovi\'{c}, J. Peters, G. Ru\v{z}i\'{c}, Efficient domination in circulant graphs with two chord lengths,  Inform. Process. Lett. 102(6) (2007) 253--258.

\bibitem{Rej13}
K. Reji Kumar, G. MacGillivray, Efficient domination in circulant graphs,  Discrete Math. 313(6) (2013) 767--771.

\bibitem{cr13}
T. Tamizh Chelvam,  S. Mutharasu, Subgroups as efficient dominating sets in Cayley graphs, Discrete  Appl. Math. 161(9) (2013) 1187--1190.

\bibitem{Wang20}
Y. Wang, B. Xia, S. Zhou, Subgroup regular sets in Cayley graphs, Discrete Math. 345(11) (2022) 113023.


\bibitem{Wang21}
Y. Wang, B. Xia, S. Zhou, Regular sets in Cayley graphs, J. Algebra Combin. 57 (2023) 547--558.


\bibitem{Zhang20}
J. Zhang, S. Zhou, On subgroup perfect codes in Cayley graphs,  European J. Combin. 91 (2021) 103228.

\bibitem{Zhang22}
J. Zhang, S. Zhou, Corrigendum to ``On subgroup perfect codes in Cayley graphs" [European J. Combin. 91 (2021) 103228], European J. Combin. 101 (2022) 103461.

\bibitem{ZZ22}
J. Zhang, S. Zhou, Nowhere-zero 3-flows in nilpotently vertex-transitive graphs, preprint, \url{https://arxiv.org/abs/2203.13575}.

\end{thebibliography}
\end{document}